\documentclass[12pt,a4paper]{amsart}

\usepackage[utf8]{inputenc}
\usepackage[english]{babel}
\usepackage[T1]{fontenc}
\usepackage{amsmath}
\usepackage{mathtools}
\usepackage{amsfonts}
\usepackage{amssymb}
\usepackage{setspace}
\usepackage{amsfonts, dsfont}
\usepackage{mathrsfs, enumerate, csquotes, color}
\usepackage{xcolor}
\usepackage{bm}
\usepackage[bb=boondox]{mathalfa}
\definecolor{vertfonce}{rgb}{0.20, 0.46, 0.25}
\definecolor{rougefonce}{rgb}{0.64, 0.09, 0.20}
\usepackage[breaklinks=true,
colorlinks=true,
linkcolor=rougefonce,
citecolor=vertfonce]{hyperref}

\theoremstyle{plain}
\newtheorem{theorem}{Theorem}[section]
\newtheorem{corollary}[theorem]{Corollary}
\newtheorem{lemma}[theorem]{Lemma}
\newtheorem{proposition}[theorem]{Proposition}
\theoremstyle{definition}
\newtheorem{definition}[theorem]{Definition}
\newtheorem{assumption}[theorem]{Assumption}

\theoremstyle{remark}
\newtheorem{remark}[theorem]{Remark}
\numberwithin{equation}{section}

\theoremstyle{definition}

\DeclarePairedDelimiterX\braket[2]{\langle}{\rangle}{#1\,\delimsize\vert\,\mathopen{}#2}

\def\eps{\varepsilon}
\def\R{{\mathbb R}}% real numbers
\def\C{{\mathbb C}}% complex numbers
\def\N{{\mathbb N}}% nonnegative integers
\mathchardef\mhyphen="2D % math hypehn
\def \Op{{\rm Op}}

\def \bA{\mathbf{A}}
\def \magSch{\mathcal{L}_{h}}
\def \tmagSch{\tilde{\mathcal{L}}_{h}}
\def \magSchbar{\check{\mathcal{L}}_{\hbar}}
\def \dom{{\rm Dom}}
\def \O{\mathscr{O}}

\def \A{\mathcal{A}}
\def \B{\mathcal{B}}

\def \cQ{\mathcal{Q}}

\def \fQ{\mathfrak{Q}}
\def \fq{\mathfrak{q}}
\def \fM{\mathfrak{M}}
\def \bz{\mathbf{z}}
\def \ct{\check{t}}
\def \ft{\mathfrak{t}}
\def \fx{\mathfrak{x}}

\def\restriction#1#2{\mathchoice
              {\setbox1\hbox{${\displaystyle #1}_{\scriptstyle #2}$}
              \restrictionaux{#1}{#2}}
              {\setbox1\hbox{${\textstyle #1}_{\scriptstyle #2}$}
              \restrictionaux{#1}{#2}}
              {\setbox1\hbox{${\scriptstyle #1}_{\scriptscriptstyle #2}$}
              \restrictionaux{#1}{#2}}
              {\setbox1\hbox{${\scriptscriptstyle #1}_{\scriptscriptstyle #2}$}
              \restrictionaux{#1}{#2}}}
\def\restrictionaux#1#2{{#1\,\smash{\vrule height .8\ht1 depth .85\dp1}}_{\,#2}}

\title[]{Low-energy eigenstates in a vanishing magnetic field}
\author[L. Benedetto]{Lino Benedetto}
\address[L. Benedetto]{DMA, École normale supérieure, Université PSL, CNRS, 75005 Paris, France \& Univ Angers, CNRS, LAREMA, SFR MATHSTIC, F-49000 Angers, France} 
\email{lbenedetto@dma.ens.fr}

\numberwithin{equation}{section}

\begin{document}

\begin{abstract}
    This paper is dedicated to the spectral analysis of the semiclassical purely magnetic Laplacian $\mathcal{L}_h$, $h>0$, on the plane $\R^2$ in the situation where the magnetic field $B$ vanishes nondegenerately on an open smooth curve $\Gamma$. We prove the existence of a discrete spectrum for energy windows of the scale $h^{4/3}$ and give complete asymptotics in the semiclassical paramater $h$ for eigenvalues in such windows. Our strategy relies on the microlocalization of the corresponding eigenfunctions close to the zero locus $\Gamma$ and, following the recent paper \cite{BFKRV}, on the implementation of a Born-Oppenheimer strategy through the use of operator-valued pseudodifferential calculus and superadiatic projectors. This allows us to reduce our spectral analysis to that of effective semiclassical pseudodifferential operators in dimension 1 and apply the well-known semiclassical techniques {\it à la} Helffer-Sjöstrand.
\end{abstract}

\maketitle

\tableofcontents

\newpage

\section{Motivation and results}
\label{sect:intro}

% Notes pour l'introduction: 
% - Montgomery: sa courbe d'annulation est toujours compacte; a écrit une partie du lien entre sous-laplacien et géométrie sous-riemannienne avec les opérateurs magnétiques; ses asymptotiques spectrales se limitent au premier mode et fond du puits;
% - Raymond (breaking zero locus): exemple où l'invariance par translation n'est plus réalisée, d'où apparition de bound states;

% - Raymond, Bonnaillie-Noël, Hérau: étude spectrale des opérateurs magnétiques partiellement semiclassique par une méthode de Born-Oppenheimer et constructions WKB; pour avoir des asymptotiques à tout ordre des premieres valeurs propres, ils doivent faire leurs constructions WKB puis estimer le spectral gap (long et difficile): il faut faire des commentaires sur la différence de stratégie; emphase mise sur le spectral gap: celui-ci apparaît aux ordres inférieures (en h) et à cause de la géométrie, i.e. courbure; remarque 1.11, ils ont des puissances 1/2 de h, j'ai pas mieux?

% - Kordyukov: papier sur spectral gaps for periodic magnetic Schrödinger, j'ai l'impression qu'il considère des champs magnétiques avec des zéros isolés (mais il parle aussi de cas d'annulation plus général); intéressant mais je ne suis pas sûr comment arriver à le relier à mon travail;

% - Raymond, Fahs, Vu Ngoc, Le Treust: relire l'introduction;

% - Raymond, Dombrowski: je dois bien faire référence à cet article et expliquer son rapport avec mon travail;

\subsection{About the magnetic Laplacian with vanishing magnetic field}
This paper is dedicated to the description of the spectrum of the semiclassical magnetic Laplacian $$\mathcal{L}_{h,\bA} := (-ih\nabla - \bA)^2,\quad h>0,$$ on the whole plane $\R^2$. Here the vector potential $\bA: \R^2 \rightarrow \R^2$ is supposed to be smooth and generating a magnetic field $B = \nabla\times \bA$ which vanishes nondegenerately on a curve $\Gamma$, i.e. 
\begin{equation*}
    \Gamma = \{B = 0\}.
\end{equation*}
More precisely, the unbounded operator $\mathcal{L}_{h,\bA}$ is defined as the unique self-adjoint extension of the operator given by the quadratic form
\begin{equation*}
    \forall \varphi \in C_c^\infty(\R^2),\quad \mathcal{Q}_{h,\bA}(\varphi) := \int_{\R^2} |(-ih\nabla - \bA)\varphi(z)|^2 dz.
\end{equation*}
Its domain is then simply given by
\begin{equation*}
    \dom(\mathcal{L}_{h,\bA}) = \{\varphi \in H_\bA^{1}(\R^2)\,:\,\mathcal{L}_{h,\bA}\varphi \in L^2(\R^2)\}.
\end{equation*}

We are interested in the spectral theory of the operator $\mathcal{L}_{h,\bA}$ as the semiclassical parameter $h$ goes to $0$. In contrast with the context of an electric potential $V\in C^\infty(\R^2)$ bounded by below, where its vanishing has no consequence on the spectral analysis of its associated Schrödinger operator $-h^2\Delta + V$ (up to a fixed translation), the zero locus of the magnetic field $B$ plays a significant role, both spectrally and classically. This observation was first made by Montgomery in \cite{Mon} where, in the setting of a compact curve $\Gamma$, the concentration at the zero locus of $B$ of the first eigenmodes was established, as well as the first-order asymptotics in $h$ for the corresponding eigenvalues. Coming from motivations in subRiemannian (sR) geometry, Montgomery proved that the zero locus of the magnetic field, which gives rise to \emph{singular curves} in sR geometry, persists under quantization. 

More precise asymptotics of the first eigenvalue was given by Pan and Kwek in \cite{PK} and its role in the description of the nucleation phenomenon for superconductors subject to non-homogeneous magnetic fields. In \cite{HelKor}, Helffer and Kordyukov extended this first-order asympotic eigenvalue to higher dimensions with a vanishing of the magnetic field on a hypersurface and considered the problem of spectral gaps in between the first eigenvalues. Complete asymptotics of the first eigenvalues of $\mathcal{L}_{h,\bA}$ were obtained by Dombrowski and Raymond (\cite{DR}), as well as by Bonnaillie-Noël, Hérau and Raymond (\cite{BonRayHer}). In both these works, the strategy consists in producing good quasimodes and estimating the spectral splitting between the first eigenvalues by localization and reduction of dimension using Feshback-Grushin projections (or corrected versions of it).
For higher eigenvalues, Weyl laws for semiexcited states were obtained by Keraval (\cite{Ker}).

\vspace{0.25cm}

The aim of this article is to extend the previous results using an alternative approach developed in \cite{BFKRV} and to give a complete asymptotic description, i.e. up to order $\O(h^\infty)$, of the lower part of the discrete spectrum of the magnetic Laplacian $\mathcal{L}_{h,\bA}$, in particular beyond the semiexcited regime.

\subsection{Confinement and localization close to the zero locus}

As stated above, the zero locus of the magnetic field $B$ plays a crucial role in the spectral theory of the magnetic Laplacian. In this article, we consider the case of an open and non-self-intersecting curve
$$\Gamma = \{\gamma(x)\,:\,x\in\R\} \subset \R^2,$$ 
going to infinity. This choice is in contrast to most of the previous works in which the zero locus was supposed to be compact. We refer to Section \ref{subsect:locdescSch} for the precise assumptions made about $\Gamma$.

\subsubsection{Existence of a discrete spectrum}

In this setting, it is not straightforward that the self-adjoint operator $\mathcal{L}_{h,\bA}$ has any discrete part in its spectrum. Indeed, it is known that within the framework of magnetic Laplacians, the absolute value $|B|$ of the magnetic field usually plays the role of a confining potential and allows for so-called "magnetic bottles" (\cite{HeMo}). Here, as $\Gamma$ goes to infinity, this is not enough to conclude to the existence of a discrete spectrum. This issue was considered in \cite{BonRay}, \cite{Ray}, where confinement is induced by non-smooth zero loci and in \cite{BonRayHer} in the case of a straight line zero locus. In our present situation, since $\Gamma$ is smooth but not necessarily straight, we will observe that under some further assumptions made precise in Section \ref{spectral}, we can single out the transverse derivative $\partial_{\vec n} B$ of the magnetic field on $\Gamma$ as playing the role of a confining potential.  

\begin{theorem}
    \label{thm:spectrediscret}
    Under Assumptions \ref{assum:uniformtubular}-\ref{assum:kappa}, there exists $E_0>0$ such that, for any $E \in (-\infty,E_0)$, one can find $h_0 > 0$ such that for $h\in(0,h_0)$ the spectrum of $\magSch$ in the energy window $(-\infty,Eh^{4/3}]$ is purely discrete.
\end{theorem}

With this theorem in hand, one can immediately adapt the proof of \cite{DR},\cite{Mon} to obtain the following Agmon-type localization result for eigenfunctions corresponding to eigenvalues in the energy window $(-\infty, Eh^{4/3}]$.

\begin{theorem}[\cite{DR},\cite{Mon}]
\label{thm:firstlocgamma}
    Let $E \in (-\infty,E_0)$. There exist $C>0$, $\alpha >0$ and $h_0 > 0$ such that, for $h\in(0,h_0 ]$ and for any eigenpair $(\mu_h, \psi_h)$ of $\mathcal{L}_{h,\bA}$ satisfying $\mu_h \leq Eh^{4/3}$, we have
    \begin{equation*}
        \int_{\R^2} e^{2\alpha h^{-1/3}{\rm dist}(z,\Gamma)} |\psi_h (z)|^2 \,dz \leq C \Vert \psi_h \Vert_{L^2(\R^2)}^2.
    \end{equation*}
\end{theorem}

This last estimate shows that, for the energy window we are interested in, we are led to localize our analysis close to the curve $\Gamma := \{\gamma(x),~x\in\R\}$ on the scale $$\hbar = h^{1/3}.$$

\subsubsection{Tubular coordinates and heuristic}

This leads to considering tubular coordinates $(x,t)$ on a neighborhood of the zero locus by considering the map
\begin{equation*}
    \forall (x,t)\in\R^2,\quad \Phi(x,t) = \gamma(x) + t\, \vec n(x),
\end{equation*}
where $\vec n(x)$ denotes the unitary vector field normal to $\Gamma$ at $\gamma(x)$ satisfying
\begin{equation*}
    {\rm det}(\gamma'(x),\vec n(x)) = 1,~x\in\R.
\end{equation*} 
The curvature $k(x)$ at the point $\gamma(x)$ is given in this parametrization by the relation
\begin{equation*}
    \gamma''(x) = k(x)\vec n(x).
\end{equation*}

More precisely, we consider \emph{rescaled tubular coordinates} $(x,\ct)=(x,\hbar t)$ on a neighborhood of $\Gamma$ (see Section \ref{subsect:locdescSch}). For the sake of clarity, we adopt here a heuristic point of view in order to introduce the leading operator of the analysis presented in this paper.
In a first approximation, near a point $(x,0)\in \Gamma$, the zero locus can be assumed to be straight, that is, given by $\ct = 0$, and that the magnetic field cancels linearly so that we can write $B(x,\ct) = \delta(x)\ct + \O(\ct^2)$ where $\delta(x)$ is the derivative of $B$ with respect to $\ct$. Therefore, by careful consideration of the homogeneity with respect to the semiclassical parameter $h>0$, at leading order near $\Gamma$, the operator to consider is given by
\begin{equation}
    \label{eq:heuristic}
    h^{4/3}\left[D_{\ct}^2+\left(\hbar D_x - \frac{\delta(x)}{2}\ct^2\right)^2\right].
\end{equation}

\subsection{Montgomery operators}
\label{subsect:Montgomery}
This heuristic led previous works to consider the self-adjoint realizations on $\R$ of the following differential operators
\begin{equation*}
    \fM(\nu) = D_{\ft}^2 + \left(\nu - \frac{1}{2} \ft^2 \right)^2,\quad\ft\in\R,
\end{equation*}
with parameter $\nu \in \R$, called the {\it Montgomery operators} (\cite{Mon},\cite{PK},\cite{HelKor},\cite{DR}), whose essential properties we recall here. Their domain, endowed with the graph norm, is given by the Hilbert space
\begin{equation*}
    B^{2,4}(\R) = \{u\in L^2(\R)\,:\, u''\in L^2(\R),\  \ft^4 u \in L^2(\R)\}.
\end{equation*}
For any $\nu\in\R$, the operator $\fM(\nu)$ has compact resolvent and thus have discrete spectrum:
\begin{equation*}
    \sigma(\fM(\nu)) := \{\tilde{\mu}_{1}(\nu)\leq \tilde{\mu}_{2}(\nu)\leq \dots\}.
\end{equation*}
For $k\geq 1$, the map $\nu \in\R\mapsto \tilde{\mu}_k(\nu)$ is called the $k$-th \emph{dispersive curve}.

\vspace{0.25cm}

The family $(\fM(\nu))_{\nu\in\R}$ has been studied by many authors; see, for example, \cite{HL},\cite{HP}. 
In particular, it is an analytic family of type B in the sense of Kato (\cite{Kato}), and, as a consequence, every dispersive curve $\nu\in\R\mapsto \tilde{\mu}_{k}(\nu)$, $k\geq1$, is smooth.
Moreover, since one can prove that the eigenvalues of the Montgomery operators are simple, it is possible to single out an eigenfunction associated to the $k$-the dispersive curve
\begin{equation}
    \label{eq:basisMon}
    \nu\in\R\mapsto \tilde{u}_k(\nu)\in L^2(\R_{\ft}),\quad k\geq 1,
\end{equation}
normalized in $L^2(\R_\ft)$ and depending smoothly on the parameter $\nu$. It follows from usual considerations in the spectral theory of Schrödinger operators with electric potential that each eigenfunction $\tilde{u}_k(\nu)$, for $k\geq 1$ and $\nu\in\R$, is smooth and decreases exponentially with respect to the variable $\ft\in\R$.

\vspace{0.25cm}

The main focus in the literature concerning this family of differential operators has been on the existence and uniqueness of critical points of the dispersive curves.
We recall here the most recent result on this question.

\begin{theorem}[\cite{HL}, Theorem 1.5]
    \label{thm:HLmontgomery}
    There exists $k_0\geq1$ such that for $k=1$ or $k\geq k_0$, $\tilde{\mu}_k$ admits a unique critical point $\nu_{k,c}$. Moreover, it corresponds to a global nondegenerate minimum.
\end{theorem}

In what follows, we denote by $\tilde{\mu}_{k,c} := \tilde{\mu}_k(\nu_{k,c})$ the corresponding minimal value of the $k$-th dispersive curve, as well as by $\tilde{u}_{k,c} := \tilde{u}_k(\nu_c)$ the corresponding eigenfunction. Theorem \ref{thm:HLmontgomery} is conjectured to hold for all dispersive curves of the Montgomery operators. Note also that, for any $k\geq 1$, $\nu\mapsto\tilde{\mu}_{k}(\nu)$ admits at least one critical point 
as it diverges (polynomially) to $+\infty$ as $\nu$ goes to $\pm \infty$ (\cite{HL}).

\subsection{Born-Oppenheimer strategy}

Returning to the leading operator given in Equation \eqref{eq:heuristic}, we observe that we are dealing with a partially semiclassical differential operator. This situation, due to the anisotropy of $\mathcal{L}_{h,\bA}$, is a trademark in the theory of magnetic Laplacians and is even considered in \cite{BonRayHer} as the starting point of their analysis. Following their approach, we introduce operator-valued symbols and their associated $\hbar$-semiclassical quantization (see Appendix \ref{sect:pseudoropvalued}).

\subsubsection{Operator-valued pseudodifferential calculus}

Explicitly, the heuristic described above is rigorously implemented by first microlocalizing on a compact subset $K_E$ of the phase space $T^*\Gamma$ the eigenfunctions corresponding to eigenvalues in the energy window $(-\infty,Eh^{4/3}]$. This leads us to consider a symbol $n_{\hbar} \in S^0(\R^2,\mathscr{L}(B^{2,4}(\R),L^2(\R)))$ whose principal symbol is given by
\begin{equation}
    \forall (x,\xi)\in\R^2,\quad n_{0}(x,\xi)  = D_{\ct}^2 + \left(\mathring{\Xi}(\xi) - \frac{\mathring{\delta}(x)}{2}\ct^2\right)^2,
\end{equation}
where $\mathring{\Xi}$ and $\mathring{\delta}$ are both bounded truncations of $\xi$ and $\delta$ respectively, coinciding with them on the compact set $K_E$ (see Section \ref{subsect:reducbounded}). We denote by $\mathcal{N}_\hbar$ the quantization of the symbol $n_\hbar^w$ acting on $L^2(\R,B^{2,4}(\R))$.

The Born-Oppenheimer strategy then consists of diagonalizing the principal symbol $n_0$ and considering each eigenvalue individually to effectively reduce the dimension. 
To this end, we observe that for all $(x,\xi)\in\R^2$, $n_{0}(x,\xi)$ is a self-adjoint operator with compact resolvent and thus has a discrete spectrum
\begin{equation*}
    \sigma(n_0(x,\xi)) := \{\mathring{\mu}_{1}(x,\xi)\leq \mathring{\mu}_{2}(x,\xi)\leq \dots\}.
\end{equation*}
This is where the family of Montgomery operators comes into play. The spectral theory of $n_0(x,\xi)$, $(x,\xi)\in\R^2$, and $(\fM(\nu))_{\nu\in\R}$ are related by the following rescaling: writing $$\ft = \mathring{\delta}(x)^{1/3}\ct,$$ and considering the associated unitary operator $\mathcal{T}(x):L^2(\R_{\ct})\rightarrow L^2(\R_\ft)$ defined by 
\begin{equation*}
    \forall f\in L^2(\R_{\ct}),\quad \mathcal{T}(x)f(\ft) = \mathring{\delta}(x)^{1/6}f(\mathring{\delta}(x)^{1/3}\ct),
\end{equation*}
the operator $n_{0}(x,\xi)$ is intertwined trough the unitary $\mathcal{T}$ with
\begin{equation}
    \label{eq:n0rescaling}
   \mathring{\delta}(x)^{2/3}\left(D_{\ft}^2 + \left(\mathring{\Xi}(\xi)\mathring{\delta}(x)^{-1/3} - \frac{1}{2} \ft^2 \right)^2\right) = \mathring{\delta}(x)^{2/3}\fM(\mathring{\Xi}(\xi) \mathring{\delta}(x)^{-1/3}).
\end{equation}
In particular, we obtain this way that each eigenvalue of $n_0(x,\xi)$ is simple and that, for all $k\geq 1$,
\begin{equation}
    \label{eq:smootheigenvalues}
    \mathring{\mu}_{k}(x,\xi) = \mathring{\delta}(x)^{2/3}\tilde{\mu}_{k}(\mathring{\Xi}(\xi)\mathring{\delta}(x)^{-1/3}).
\end{equation}

Using Equation \eqref{eq:basisMon}, the same rescaling allows us to single out an eigenvector for each $n\geq 1$ that depends smoothly on $(x,\xi)\in\R^2$ by setting:
\begin{equation}
    \label{eq:smootheigenfunction}
    \mathring{u}_k(x,\xi) = \mathcal{T}^*(x)\tilde{u}_k(\mathring{\Xi}(\xi)\mathring{\delta}(x)^{-1/3}).
\end{equation}

These spectral properties satisfied by the principal symbol $n_0$ allow us to consider the symbols $\Pi_0^k \in S^0(\R^2,\mathscr{L}(L^2(\R),B^{2,4}(\R)))$, $k\geq 1$, defined at each point $(x,\xi)\in\R^2$, as the orthogonal projection on the subspace generated by the eigenfunction $\mathring{u}_k(x,\xi)$.

\subsubsection{Superadiabatic projectors}

In order to effectively implement the Born-Oppenheimer strategy, we deviate however from previous works and follow the approach proposed by \cite{BFKRV}, that is, through the construction of approximate stable subspaces (up to $\O(\hbar^\infty)$) using the symbolic properties of the operator-valued semiclassical pseudodifferential calculus. In anticipation of Section \ref{subsect:reductionscheme}, this construction results in the existence of a symbol $\Pi_\hbar^k\in S^0(\R^2,\mathscr{L}(L^2(\R),B^{2,4}(\R)))$ associated with the eigenvalue $\mu_k$, $k\geq1$, satisfying
\begin{equation*}
    \Pi_\hbar^k = \Pi_0^k + \O(\hbar),
\end{equation*}
as well as
\begin{equation*}
    \Pi_\hbar^{k,w}\circ \Pi_\hbar^{k,w}=\Pi_\hbar^{k,w}+\mathscr{O}(\hbar^\infty)\quad\mbox{and}\quad [\mathcal{N}_\hbar,\Pi_\hbar^{k,w}]=\mathscr{O}(\hbar^\infty).
\end{equation*}
The operators $\Pi_\hbar^{k,w}$, $k\geq1$, are called \emph{superadiabatic projectors}, and allow us to look for quasimodes and study the spectrum of $\mathcal{N}_\hbar$ by restriction to each corresponding (approximate) stable subspace ${\rm Ran}(\Pi_\hbar^{k,w})$.

\subsection{Main results}

In order to state our results concerning the spectral theory of the magnetic Laplacian $\mathcal{L}_{h,\bA}$, we make precise the definition of equivalence of spectra of two operators.

\begin{definition}[\cite{FLTRVN}]
    \label{def:spectracoincide}
    The spectra of two self-adjoint operators $T_1$ and $T_2$ depending on $h$ are said to coincide in an interval $I_h$ modulo $\O(h^\alpha)$, 
    $\alpha\in\R\cup \{+\infty\}$, when there exists $C, h_0 >0$ such that, for all $h\in(0,h_0)$,
    \begin{enumerate}
        \item $T_1$ and $T_2$ have discrete spectrum in $I_h + [-Ch^\alpha,C h^\alpha]$;
        \item for all interval $L_h \subset I_h$, we can find an interval $K_h$ such that $L_h\subset K_h$ with ${\rm d}_{H}(K_h,L_h)\leq C h^\alpha$ and 
        \begin{equation*}
            {\rm rank} \mathbb{1}_{L_h}(T_1)\leq {\rm rank} \mathbb{1}_{K_h}(T_2)\quad\mbox{and}\quad {\rm rank} \mathbb{1}_{L_h}(T_2)\leq {\rm rank} \mathbb{1}_{K_h}(T_1),
        \end{equation*}
        where ${\rm d}_{H}$ denotes the Hausdorff distance:
        \begin{equation*}
            {\rm d}_{H}(A,B) = \sup_{(a,b)\in A\times B}\max(d(a,B),d(b,A)).
        \end{equation*}
    \end{enumerate}
\end{definition}

\subsubsection{Dimensional reduction}

We can now state our main result where we use, among others, the eigenvalues $\mathring{\mu}_k$ and their corresponding eigenfunctions $\mathring{u}_k$.

\begin{theorem}
    \label{thm:main}
    Under Assumptions \ref{assum:uniformtubular}-\ref{assum:kappa}, let $E\in(-\infty,E_0)$ as in Theorem \ref{thm:spectrediscret}. Then, there exists $k_E \geq 1$ such that the spectrum of $\mathcal{L}_{h,\bA}$ in $I_h = (-\infty,E h^{4/3})$ coincides modulo $\O(h^\infty)$ with that of a bounded operator of the form
    \begin{equation*}
        h^{4/3}\begin{bmatrix}
            \mathring{\mu}_{\hbar,1}^{w} & 0 & \cdots & 0 \\
            0 & \mathring{\mu}_{\hbar,2}^{w} &  & \vdots \\
            \vdots &  &  \ddots & 0 \\
            0 & \cdots & 0 & \mathring{\mu}_{\hbar,k_E}^{w}
         \end{bmatrix},
    \end{equation*}
    acting on $L^2(\R,\C^{k_E}$), where for $k\in\{1,\dots,k_E\}$, the \emph{effective Hamiltonian} $\mu_{\hbar,k}^w$ is a (uniquely defined) $\hbar$-semiclassical self-adjoint pseudodifferential operator with admissible symbol
     \begin{equation*}
        \mathring{\mu}_{\hbar,k} \sim \mathring{\mu}_{k} + \sum_{j\geq 1}\hbar^j \mathring{\mu}_{j,k},\quad \mbox{with}\quad \forall j\geq 1,\ \mathring{\mu}_{j,k}\in S^0(\R^2).
    \end{equation*}
    and where the subprincipal symbol $\mu_{1,k}$ is given by
    \begin{equation*}
        \mathring{\mu}_{1,k} = 2k\langle \mathring{u}_k, p_0^2 \ct \,\mathring{u}_k\rangle- 2 \mathring{\kappa}\langle \mathring{u}_k, p_0 \ct^3 \mathring{u}_k\rangle +  {\rm Im} \langle \mathring{u}_k, \{n_0,\mathring{u}_k\}\rangle + \mu_k {\rm Im} \langle \partial_x \mathring{u}_k,\partial_\xi \mathring{u}_k\rangle.
    \end{equation*}
    where $\mathring{\kappa}(x) = \frac{1}{6}\partial_{t}^{2}\tilde{B}(x,0) - \frac{1}{3}k(x)\mathring{\delta}(x)$ and $p_0(x,\xi) = \mathring{\Xi}(\xi) - \frac{\mathring{\delta}(x)}{2}\ct^2$.
\end{theorem}

Theorem \ref{thm:main} is a diagonalization result since it reduces the spectral analysis of the magnetic Laplacian $\mathcal{L}_{h,\bA}$ to that of a finite family of semiclassical pseudodifferential operators in one dimension: the spectrum of $\mathcal{L}_{h,\bA}$ is the superposition (counting multiplicities) of the spectra of $\mathring{\mu}_{\hbar,k}^w$, $k\in\{1,\dots,k_E\}$.

\subsubsection{Description of the spectrum in a nondegenerate well}

In order to give a complete description of the spectrum of these effective Hamiltonians, we make the following assumption on the transverse derivative of the magnetic field.

\begin{assumption}
    \label{assum:uniquewell}
    The map $x\in\R\mapsto \delta(x)$ admits a unique global nondegenerate minimum $\delta_{c}$ attained at 0.
\end{assumption}

This allows to apply directly the known results in the literature (\cite{DimSj},\cite{HeSjII},\cite{HeRo}) on the spectral theory of one-dimensional semiclassical pseudodifferential operators in the situation of a unique nondegenerate potential well. We first take a look at the bottom of the spectrum of the magnetic Laplacian in the semiexcited regime.

\begin{corollary}
    \label{cor:bottomsemiexcited}
    Let $\eta >0$. Under Assumptions \ref{assum:uniformtubular}-\ref{assum:kappa} and Assumption \ref{assum:uniquewell}, denoting by $\lambda_n(h)$, the $n$-th eigenvalue of $\mathcal{L}_{h,\bA}$, $n\in\N$, in the energy window $(-\infty,\delta_c^{2/3}\tilde{\mu}_{1,c}h^{4/3} + h^{4/3+\eta}]$, then there exists $h_0>0$ and a real valued smooth function
    \begin{equation*}
        f(\sigma;h) \sim f_0(\sigma)+hf_1(\sigma)+\dots,\quad \sigma\in\R,\ h\in (0,h_0),
    \end{equation*}
    converging in the smooth topology, such that
    \begin{equation}
        \label{eq:asympgenstuct1}
        \lambda_n(h) = h^{4/3} f\left(\hbar\left(n+\frac{1}{2}\right);\hbar\right) + \O(h^\infty),\ n\in\N,\ h\in (0,h_0),
    \end{equation}
    where the remainder $\O(h^\infty)$ is uniform for all eigenvalues in the spectral window.
    Moreover, the first terms of the asymptotic are given by
    \begin{equation}
        \label{eq:asymptprecise}
        \lambda_n(h) = h^{4/3}\left(\delta_c^{2/3}\tilde{\mu}_{1,c} + \hbar \delta_c^{2/3}\left(\langle L \tilde{u}_{1,c},\tilde{u}_{1,c}\rangle + (2n+1)c_1\right) + \O(\hbar^2) \right),
    \end{equation}
    where we have set
    \begin{equation*}
        L = 2\delta_c^{-4/3}\mathring{\kappa}(0)\left(\frac{\ft^2}{2}-\nu_{1,c}\right)\ft^3 + 2 \delta_c^{-1/3}k(0)\left(\nu_{1,c}-\frac{\ft^2}{2}\right)^2\ft,
    \end{equation*}
    and
    \begin{equation*}
        c_1 = \left(\frac{\alpha \,\tilde{\mu}_{1,c}\, \partial_\nu^2\tilde{\mu}_{1}(\nu_{1,c})}{2}\right)^{1/2}, \quad \alpha = \frac{1}{2}\delta_c^{-1}\delta''(0) >0.
    \end{equation*}
\end{corollary}

This result is consistent with the asymptotics found for the first eigenvalues of $\mathcal{L}_{h,\bA}$ in \cite{DR}. However, it should be noted that, beyond the fact that we can deal with semiexcited states, Corollary \ref{cor:bottomsemiexcited} is an improvement since the asymptotic presented here is in powers of $\hbar = h^{1/3}$, whereas the one given in \cite{DR} is in powers of $h^{1/6}$.

A similar result also holds for any effective Hamiltonian $\mathring{\mu}_{\hbar,k}$, $k\geq 1$, as long as the conclusion of Theorem \ref{thm:HLmontgomery} is satisfied by the $k$-th dispersive curve $\nu\in\R\mapsto\tilde{\mu}_k(\nu)$.

\vspace{0.25cm}

To go beyond the semiexcited regime, our second application concerns energy windows of the form $[E_1 h^{4/3}, E_2 h^{4/3}]$, $E_1, E_2 < E_0$, where the interval $I = [E_1,E_2]$ contains no critical values for the effective Hamiltonians.

\begin{assumption}
    \label{assum:regularvalues}
    There exists $\eps >0$ such that $d\mathring{\mu_k} \neq 0$, for all $(x,\xi)\in \mathring{\mu}_k^{-1}(I_\eps)$, for $k\in\{1,\dots,k_I\}$, where $I_\eps = [E_1-\eps,E_2+\eps]$ and $k_I = k_{E_2}$.
\end{assumption}

In this situation, the description of the spectrum is given by an extension to all orders of the Bohr-Sommerfeld rules (\cite{HeRo1},\cite{HeRo}). To state the result, we introduce for $k\in\{1,\dots,k_I\}$ the smooth map
\begin{equation*}
    \forall E\in I_\eps,\quad J_k(E) := \int_{\gamma_k(E)}\xi dx,
\end{equation*}
where $\gamma_k(E)$ is the periodic bicharacteristic associated with the Hamiltonian $\mathring{\mu}_k$ and the energy level $E$. Keeping the same notation for them, we extend the maps $J_k$, $k\in\{1,\dots,k_I\}$ to the real line $\R$ in such a way that they are increasing. In particular, they are smooth diffeomorphisms from $[E_1,E_2]$ onto their images. It then follows from \cite{DimSj}[Lemma 5.2] that the following quantities
\begin{equation*}
    \beta_k := \frac{1}{2\pi}\int_{J_K(\gamma_k(E))}\xi dx -J_k(E), \quad E\in I_\eps,
\end{equation*}
are constants. By direct application of \cite{DimSj}[Theorem 15.10], we obtain the following result.

\begin{corollary}
    \label{cor:spectrumexcited}
    Under Assumptions \ref{assum:uniformtubular}-\ref{assum:kappa} and Assumptions \ref{assum:uniquewell}-\ref{assum:regularvalues}, there exist $h_0>0$ and, for each $k\in \{1,\dots,k_I\}$, a smooth map $J_k(I_\eps) \ni \sigma \rightarrow g^k(\sigma;h)\in\R$ with an asymptotic expansion
    \begin{equation*}
        g^k(\sigma;h) \sim \sigma + hg_1^k(\sigma)+\dots,\quad \sigma \in J(I_\eps),\ h\in (0,h_0),
    \end{equation*}
    converging in the smooth topology, such that the spectrum of the magnetic Laplacian $\mathcal{L}_{h,\bA}$ in the energy window $[E_1 h^{4/3},E_2 h^{4/3}]$ coincides, modulo $\O(h^\infty)$, with the disjoint union
    \begin{equation*}
        h^{4/3}\left(\bigsqcup_{k=1}^{k_I} J_{k}^{-1}\left(\{g^k(\sigma;\hbar)\,:\,\sigma \in \left(\hbar(\N+\frac{1}{2})+\beta_k\right)\cap J_k(I_\eps)\}\right)\right)\cap [E_1 h^{4/3},E_2 h^{4/3}].
    \end{equation*}
\end{corollary}

\subsection{Organization of the article}

In Section \ref{spectral}, by means of the tubular coordinates $(x,t)$ and the rescaling $\ct = h^{1/3}t$, we present a local normal form of the magnetic Laplacian close to the curve $\Gamma$, introduced by \cite{DR}. From this normal form, we deduce Theorem \ref{thm:spectrediscret}, whose proof is inspired by \cite{BonRayHer}. In Section \ref{sect:microlocalization}, we analyze the microlocalization properties of the eigenfunctions associated with the eigenvalues contained in the energy window $(-\infty, E_0 h^{4/3})$. Building on Theorem \ref{thm:firstlocgamma}, we further prove microlocalization in a compact region of $T^*\Gamma$. This allows us to reduce the spectral analysis of the magnetic Laplacian to that of the bounded pseudodifferential operator $\mathcal{N}_\hbar$, briefly introduced above, with an operator-valued symbol $n_\hbar$. The Born-Oppenheimer dimensional reduction is done in Section \ref{sect:dimreduction}: we recall there the general scheme developed in \cite{BFKRV} and apply it to the operator $\mathcal{N}_\hbar$ to prove Theorem \ref{thm:main}. The spectral analysis of these effective Hamiltonians is conducted in Section \ref{subsect:onedimensionpseudos} and we prove there Corollary \ref{cor:bottomsemiexcited}. In Appendix \ref{sect:pseudoropvalued}, we recall the main definitions and results concerning pseudodifferential calculus in $\R^n$ with operator-valued symbols.

%%%%%%%%%%%%%%%%%%%%%%%%%%%%%%%%%%%%%%%%%%%%%%%%%%%%%%%%%%%%%%%%%%%%%%%%%%%%%%%%%%%%%%%%%%%%%%%%%%%%%%%%%%%%%%%%%%%%%%%%%%%%%%%%%%%%%%%%%%

\section{Confinement and discrete spectrum}
\label{spectral}

\subsection{Towards a local normal form}
\label{subsect:locdescSch}

We give an explicit description of the magnetic Laplacian $\mathcal{L}_{h,\bA}$ close to the curve $\Gamma = \{\gamma(x),~x\in\R\}$, which we suppose parametrized by arc length.

\subsubsection{Local tubular coordinates}
As introduced in Section \ref{sect:intro}, we consider the standard tubular coordinates $(x,t)$ on a neighbourhood of the curve by considering the map
\begin{equation}
    \label{eq:standardtubucoord}
    \forall (x,t)\in\R^2,\quad \Phi(x,t) = \gamma(x) + t\, \vec n(x).
\end{equation}

In what follows, we will assume that there exists a uniform distance from the curve $\Gamma$ for which the previous coordinates make sense.

\begin{assumption}
    \label{assum:uniformtubular}
    There exists $d_0 > 0$ such that the tubular coordinates 
    \begin{equation*}
        \begin{aligned}
        \Phi\,:\,&\R\times (-d_0,d_0) \rightarrow \R^2\\
        &(x,t) \mapsto \gamma(x) + t\,\vec n(x),
        \end{aligned}
    \end{equation*}
    define a diffeomorphism and we denote by $\Omega_0 = \Phi(\R \times (-d_0,d_0))$ the corresponding tubular neighborhood.
\end{assumption}

Observe that the Jacobian of the diffeomorphism $\Phi$, that we denote by $m$, is given by 
\begin{equation*}
    \forall (x,t)\in\R\times (-d_0,d_0),\quad m(x,t) = 1 - tk(x).
\end{equation*}

Assumption \ref{assum:uniformtubular} implies in particular that the curvature $k$ must be bounded and we denote by $K>0$ such a bound. Similarly, it implies the existence of $m_0 >0$ such that $m \geq m_0$ on $\R\times(-d_0,d_0)$.

\subsubsection{Magnetic Laplacian in tubular coordinates}

We now give a description of the magnetic Laplacian in the tubular coordinates. The pull-back by the diffeomorphism $\Phi$ of the magnetic potential $\bA$, denoted by $(\tilde{A}_1,\tilde{A}_2)$, is given for $(x,t)\in\R\times(-d_0,d_0)$ by
\begin{equation*}
    \tilde{A}_1(x,t) = (1-tk(x)) \,\bA(\Phi(x,t)) \cdot \gamma'(x),\quad 
    \tilde{A}_2(x,t) = \bA(\Phi(x,t)) \cdot \vec n(x).
\end{equation*}
Introducing the notation $\tilde{B} = B\circ \Phi$, we get 
\begin{equation*}
    \partial_x \tilde{A}_2 - \partial_t \tilde{A}_1 = m\tilde{B}.
\end{equation*}

The quadratic form $\mathcal{Q}_{h,\bA}$ has the following expression in these coordinates:
for any $\varphi \in H^1(\R^2)$ whose support is contained in $\Omega_{0}$, writing $\psi = \varphi \circ \Phi$, we have
\begin{multline*}
    \int_{\R^2} |(-ih\nabla+\bA)\varphi|^2\,dx =\\
    \int_{\R\times(-d_0,d_0)} \left\{m(x,t)^{-2}|(-ih\partial_x + \tilde{A}_1(x,t))\psi|^2 + |(-ih\partial_t + \tilde{A}_2(x,t))\psi|^2 \right\}m(x,t)\,dxdt.
\end{multline*}
As $\R\times (-d_0,d_0)$ is simply connected, there exists a choice of gauge such that
\begin{equation*}
    \tilde{A}_1(x,t) = -\int_0^t (1-t'k(x))\tilde{B}(x,t')\,dt' \quad \mbox{and}\quad \tilde{A}_2(x,t) = 0.
\end{equation*}
In what follows, we will simply write $\tilde{A}$ for $\tilde{A}_1$ and denote by $\mathcal{L}_h$ the magnetic Laplacian for this choice of gauge.

As it is preferable to consider transformations preserving the $L^2$ norm, we introduce the unitary transformation 
\begin{equation*}
    \tilde{\varphi} = m^{-1/2} \varphi \circ \Phi,
\end{equation*} 
for all $\varphi \in H^1(\R^2)$ whose support is contained in $\Omega_{0}$. Then, thanks to the previous expression of the quadratic form, the magnetic Laplacian $\mathcal{L}_{h}$ coincide through this unitary transformation (and up to the change of gauge), 
close to the curve $\Gamma$ with the operator
\begin{equation*}
    \tmagSch = P_{1}^2 + P_{2}^2 - h^2\frac{k(x)^2}{4m^2},
\end{equation*}
where $P_1 = hD_t$ and $P_2 = m^{-1/2}(hD_x + \tilde{A}(x,t))m^{-1/2}$, defined on a domain of $L^2(\R\times(-d_0,d_0);dx dt)$. More precisely, we can write 
\begin{equation*}
    \magSch \varphi = \tmagSch \tilde{\varphi}.
\end{equation*}

\subsubsection{Behavior of the magnetic field close to its zero locus}
We now give the precise assumptions we make on our magnetic field. We first make an assumption on $B$ away from $\Gamma$.
\begin{assumption}
    \label{assum:confinement}
    We suppose that the intensity of the magnetic field $|B|$ is uniformly bounded by below outside of the tubular neighborhood $\Omega_0$: there exists a constant $b_0 > 0$ such that 
    \begin{equation*}
        \forall z\in \R^2 \setminus \Omega_0,\ |B(z)| \geq b_0.
    \end{equation*}
\end{assumption}

The following assumption describe the behavior of $B$ an the tubular neighborhood $\Omega_0$.

\begin{assumption}
    \label{increasingBnearGamma}
    We suppose that the map $(x,t)\in\R\times(-d_0,d_0)\mapsto \partial_{t}\tilde{B}(x,t)=\nabla B(\Phi(x,t))\cdot \vec n(x)$ is uniformly bounded by below: there exists $\delta_{0}>0$ such that 
    \begin{equation*}
        \forall (x,t) \in\R\times(-d_0,d_0),\ \partial_{t}\tilde{B}(x,t) > \delta_{0}.
    \end{equation*}
\end{assumption}

In what follows, we denote by $\delta$ the transversal derivative $x\in\R\rightarrow \partial_{t}\tilde{B}(x,0)\in\R$.
By the previous assumption, $\delta$ is bounded by below: we denote by $\delta_{\rm min} > 0$ its minimum. Equivalently, the last assumption states that the annulation of $B$ on the curve $\Gamma$ is uniformly nondegenerate transversely. 

As one expects to study concentration on the curve $\Gamma$, in order to further describe our magnetic Laplacian $\magSch$, we are led to perform a Taylor expansion of the vector potential near $\Gamma$, i.e. at $t = 0$. 
We have, for $(x,t)\in\R\times(-d_0,d_0)$,
\begin{equation*}
    \tilde{B}(x,t) = \delta(x)t + \partial_{t}^{2}\tilde{B}(x,0)\frac{t^2}{2}+\O(t^3),
\end{equation*}
from which we deduce 
\begin{equation*}
    \tilde{A}(x,t) = \frac{\delta(x)}{2}t^2 + \kappa(x)t^3 + \O(t^4),
\end{equation*}
where
\begin{equation*}
    \kappa(x) = \frac{1}{6}\partial_{t}^{2}\tilde{B}(x,0) - \frac{1}{3}k(x)\delta(x).
\end{equation*}

The fact that the first term of the expansion of $\tilde{A}$ is in $t^2$ leads us to consider a rescaling with respect to the transvere variable $t$. More precisely, we recall the notation
\begin{equation*}
    \hbar = h^{1/3},
\end{equation*}    
introduced in Section \ref{sect:intro} and we consider the new variable
\begin{equation*}
    \ct = \hbar t.
\end{equation*}
The coordinates $(x,\ct)$ are called the \emph{dilated tubular coordinates}. Then we have
\begin{equation}
    \label{eq:magdilatedcoord}
        D_{t} = \hbar^{-1}D_{\ct}\quad\mbox{and}\quad
        \tilde{A}(x,t) = \hbar^{2}\frac{\delta(x)}{2}\ct^2+\hbar^3 \kappa(x)\ct^3+\hbar^{4}\check{r}(x,\ct) = \hbar^{2}\check{A}_{\hbar}(x,\ct).
\end{equation}
where the remainder $\check{r(x,\ct)}$ is of order $4$ in the variable $\ct$ at each point $x$, that is, $\check{r}(x,\ct) = \O_x(\ct^4)$.

This leads us to consider the operator
\begin{equation}
    \label{eq:rescaledmagSch}
    \magSchbar = D_{\ct}^2 + m_{\hbar}^{-1/2}\left(\hbar D_{x} - \check{A}_{\hbar}(x,\ct)\right)m_{\hbar}^{-1}\left(\hbar D_{x} - \check{A}_{\hbar}(x,\ct)\right)m_{\hbar}^{-1/2}
    -\hbar^{2}\frac{k(x)^2}{4m_{\hbar}^2},
\end{equation}
where $m_{\hbar}(u,\ct) = 1-\hbar \ct k(u)$, defined on a subspace of $L^2(\R\times (-\hbar^{-1}d_0,\hbar^{-1}d_0);dx d\ct)$. Introducing the notation 
\begin{equation}
    \label{eq:checkvarphi}
    \check{\varphi} = \hbar^{1/2}\tilde{\varphi}(x,\hbar t),
\end{equation}
and denoting by $\check{\cQ}_{\hbar}$ the quadratic form associated to $\magSchbar$, we can write 
\begin{equation}
    \label{eq:checkquadratic}
    \mathcal{Q}_h (\varphi) = \hbar^{4} \check{\cQ}_{\hbar}( \check{\varphi}).
\end{equation}

\subsubsection{Behavior of the transversal derivative}
\label{subsubsect:delta}
As stated before and as our analysis will show, the transversal derivative $\delta$ will play the role of a confining potential in the study of low-energy eigenfunctions. 
For this reason, we make the following hypotheses about the behavior of $\delta$ at infinity.

\begin{assumption}
    \label{assum:delta}
    The transversal derivative $\delta$ has controlled oscillations in the sense that
    \begin{equation*}
        \exists C_\delta > 0,~\forall x \in\R,\quad|\delta '(x)| \leq C_{\delta} \delta(x)^{2/3}.
    \end{equation*}
    Moreover, there exists $\delta_{*} > \delta_{\rm min}$ such that 
    \begin{equation*}
        \liminf_{x\rightarrow \pm\infty} \delta(x) > \delta_{*}.
    \end{equation*}
\end{assumption}

The following assumption makes sure that the magnetic potential is well approximated close to the curve $\Gamma$ by the first terms of its asymptotic given by Equation \eqref{eq:magdilatedcoord}.

\begin{assumption}
    \label{assum:kappa}
    The second-order transversal derivative of the potential $\tilde{A}$ on $\Gamma$ is controlled by $\delta$, which can be stated as
    \begin{equation*}
        \exists C_{\kappa} >0,~\forall x\in\R,\quad|\kappa(x)| \leq C_{\kappa} \delta(x),
    \end{equation*}
    while the remainder $\check{r}$ is uniformly of order 4 on $\Gamma$:
    \begin{equation*}
        \exists C_{rmd},~\forall (x,\ct)\in\R\times(-\hbar^{-1}d_0, \hbar^{-1}d_0),\quad |\check{r}(x,\ct)|\leq C_{rmd} |\ct|^4.
    \end{equation*}
\end{assumption}

In order to emphasize the confining role of the transversal derivative $\delta$, we consider yet another change of coordinates. We set
\begin{equation}
    \label{eq:coordfrak}
    \mathfrak{x} = x\quad\mbox{and}\quad\mathfrak{t} = \delta(x)^{1/3} \ct.
\end{equation}
The derivatives in these new coordinates become
\begin{equation*}
    D_{\ft} = \delta(x)^{-1/3} D_{\ct} \quad\mbox{and}\quad
    D_{\fx} = D_{x} + \frac{1}{3}\delta'(x)\delta(x)^{-1}\ft D_{\ft}. 
\end{equation*}
The space $L^2(\R^2;dx d\ct)$ is sent by this change of coordinats onto $L^2(\R^2;\delta(\fx)^{-1/3}\,d\fx d\ft)$. So as to get an operator on the space  $L^2(\R^2;d\fx d\ft)$, we shall conjugate by the squared root of the weight. 
In these new canonical coordinates $(\fx,\ft)$, the operator $\magSchbar$ is unitarily equivalent to
\begin{equation}
    \label{eq:frakL}
    \mathfrak{L}_\hbar = \delta(\fx)^{2/3}D_{\ft}^2 + \mathfrak{P}_{\hbar}^{2} - \hbar^{2} \frac{k(\fx)}{\mathfrak{m}_{\hbar}},
\end{equation}
where we have set
\begin{equation*}
    \mathfrak{P}_{\hbar} = \delta^{-1/6} \mathfrak{m}_{\hbar}^{-1/2}\left(\hbar D_{\fx} - \mathfrak{A}_{\hbar} + \frac{\hbar}{3}\delta'\delta^{-1} \ft D_{\ft}\right)\mathfrak{m}_{\hbar}^{-1/2}\delta^{1/6},
\end{equation*}
as well as
\begin{equation*}
    \mathfrak{m}_{\hbar}(\fx,\ft) = 1 - \hbar\ft \delta(\fx)^{-1/3}  k(\fx)\quad\mbox{and}\quad 
    \mathfrak{A}_{\hbar}(\fx,\ft) = \check{A}_\hbar(\fx,\delta^{-1/3}(\fx)\ft).
\end{equation*}
To conclude, we introduce the notation 
$$\mathfrak{U}_\hbar: L^2(\Omega_0;dz)\rightarrow L^2(\R^2;d\fx d\ft),$$ 
the norm-preserving linear operator associated to the composition of the various changes of coordinates introduced so far. We can finally write
\begin{equation}
    \label{eq:reductionfrak}
    \magSch  = \hbar^4\, \mathfrak{U}_\hbar^*\circ  \mathfrak{L}_\hbar \circ \mathfrak{U}_\hbar,
\end{equation}
or, equivalently, introducing the notation $\fQ_{\hbar}$ for the quadratic form associated to $\mathfrak{L}_\hbar$,
\begin{equation}
    \label{eq:frakquadratic}
    \cQ_h = \hbar^4\, \fQ_\hbar\circ \mathfrak{U_\hbar}.
\end{equation}

\begin{remark}
Note that a straightforward computation provides the equation
\begin{equation*}
    \mathfrak{P}_{\hbar} = \mathfrak{m}_{\hbar}^{-1/2}\left(\hbar D_{\fx} - \mathfrak{A}_{\hbar} + \frac{\hbar}{6}\delta'\delta^{-1} (\ft D_{\ft} + D_{\ft}\ft)\right)\mathfrak{m}_{\hbar}^{-1/2}.
\end{equation*}
In the next section, we will also make use of the following Taylor expansion of $\mathfrak{A}_{\hbar}$:
\begin{equation}
    \label{eq:taylorfrakA}
    \mathfrak{A}_{\hbar}(\fx,\ft) = \frac{1}{2}\delta(\fx)^{1/3}\ft^2 + \hbar \kappa(\fx)\delta(\fx)^{-1} \ft^3 + \hbar^2 \mathfrak{r}_{\hbar}(\fx,\ft).
\end{equation}
\end{remark}

\subsection{Discrete spectrum for low-energy windows}

This section is dedicated to the proof of Theorem \ref{thm:spectrediscret}. This is done by estimating a lower bound of the essential spectrum thanks to Persson's Theorem.

\begin{theorem}[Persson,\,\cite{Pe}]
    \label{thm:persson}
    Let $\bA\in C^1(\R^2,\R^2)$ be magnetic potential and $\mathcal{L}_\bA = (-i\nabla - \bA)^2$ the corresponding essentially self-adjoint magnetic Laplacian.
    Then, the bottom of the essential spectrum is given by 
    \begin{equation*}
        \inf \sigma_{\rm ess}(\mathcal{L}_\bA) = \Sigma(\mathcal{L}_\bA),
    \end{equation*}
    where 
    \begin{equation*}
        \Sigma(\mathcal{L}_\bA) = 
        \sup_{K\subset \R^2}\left[\inf_{\| \varphi\|=1} \langle \mathcal{L}_\bA \varphi,\varphi\rangle_{L^2}\,:\,\varphi\in C_{c}^\infty(\R^2\setminus K)\right],
    \end{equation*}
    and the sets $K$ are compact sets.
\end{theorem}

Using the local description of the magnetic Laplacian close to the curve $\Gamma$ derived in the previous section,
we are able to prove Theorem \ref{thm:spectrediscret}.

\begin{proof}[Proof of Theorem \ref{thm:spectrediscret}]
    By Theorem \ref{thm:persson}, we are to estimate the quadratic form $\cQ_h$ on test function $\varphi$ with support outside growing compacts. 
    This is done by considering successively different subdomains of the plane $\R^2$. Let $E \in (-\infty,E_0)$, where we set
    \begin{equation*}
        E_0 = \delta_*^{2/3}\tilde{\mu}_{1,c}.
    \end{equation*}

    \medskip
    \textit{Confinement away from $\Gamma$.} Let $\varphi\in C_{0}^\infty(\R^2)$. By a usual inequality in the study of magnetic Laplacians (see for instance \cite[Lemma 1.4.1]{SH}), we can write
    \begin{equation}
        \label{eq:ineqmag}
        \cQ_{h}(\varphi) \geq h\int_{\R^2} |B(z)| |\varphi(z)|^2\,dz.
    \end{equation}
    This inequality, combined with Assumption \ref{assum:confinement}, allows to say that, if the support of $\varphi$ is contained in $\R^2\setminus \Omega_0$, we have 
    \begin{equation}
        \label{eq:awayfromGamma}
        \cQ_{h}(\varphi) \geq hb_0 \int_{\R^2} |\varphi(z)|^2\,dz.
    \end{equation}

    \medskip
    \textit{Confinement in the tubular neighborhood $\Omega_0$.}
    We now treat the case where the test function $\varphi$ has its support in $\Omega_0$. We can thus use the tubular coordinates and in particular Equation \eqref{eq:checkquadratic}, that we recall here
    \begin{equation*}
    \mathcal{Q}_h (\varphi) = \hbar^{4} \check{\cQ}_{\hbar}( \check{\varphi}).
    \end{equation*}
    where the notation $\check{\varphi}$ has been introduce by Equation \eqref{eq:checkvarphi}, and defines a test function whose support is 
    in $\R\times(-\hbar^{-1}d_0,\hbar^{-1}d_0)$.
    
    \smallskip
    \textit{Step 1}. First, we prove a confinement estimate relative to the variable $\ct$. Let $r>0$ and consider the following support condition for $\check{\varphi}$:
    \begin{equation*}
        {\rm supp}(\check{\varphi}) \subset \{(x,\ct) \in \R\times(-\hbar^{-1}d_0,\hbar^{-1}d_0)\,:\, |\ct| > r\}.
    \end{equation*}
    We introduce the notation
    \begin{equation*}
        \check{P}_{\hbar} = m_{\hbar}^{-1/2}\left(\hbar D_{x} - \check{A}_{\hbar}(x,\ct)\right)m_{\hbar}^{-1/2}.
    \end{equation*}
    Observe that we have
    \begin{multline*}
        [D_{\ct},\check{P}_{\hbar}]
        = [D_{\ct},m_{\hbar}^{-1/2}]\left(\hbar D_{x} - \check{A}_{\hbar}(x,\ct)\right)m_{\hbar}^{-1/2} + m_{\hbar}^{-1/2}[D_{\ct},\hbar D_{x} - \check{A}_{\hbar}(x,\ct)]m_{\hbar}^{-1/2} \\
        + m_{\hbar}^{-1/2}\left(\hbar D_{x} - \check{A}_{\hbar}(x,\ct)\right)[D_{\ct},m_{\hbar}^{-1/2}],
    \end{multline*}
    and
    \begin{equation*}
        [D_{\ct},m_{\hbar}^{-1/2}] = \frac{i}{2}\hbar k(x)m_{\hbar}^{-3/2} \quad\mbox{and}\quad
        [D_{\ct},\hbar D_{x} - \check{A}_{\hbar}(x,\ct)] =  i\partial_{\ct}\check{A}_{\hbar}(x,\ct).
    \end{equation*}
    We get
    \begin{equation*}
        [D_{\ct},\check{P}_{\hbar}] = i\partial_{\ct}\check{A}_{\hbar}m_{\hbar}^{-1} + \frac{i}{2}\hbar \left(k(x)m_{\hbar}^{-1/2} \check{P}_{\hbar} + \check{P}_{\hbar} k(x)m_{\hbar}^{-1/2}\right).
    \end{equation*}
    In what follows we write $\check{B}_\hbar = \partial_{\ct}\check{A}_{\hbar}$. We now make use of the following inequality
    \begin{equation*}
        \left|\int_{\R^2} [D_{\ct},\check{P}_{\hbar}]\check{\varphi} \,\overline{\check{\varphi}} \,dx d\ct\right| = 2\left|\int_{\R^2} {\rm Im}(D_{\ct}\check{\varphi} \cdot \overline{\check{P}_{\hbar}\check{\varphi}})\,dx d\ct\right| 
        \leq \int_{\R^2} |D_{\ct}\check{\varphi}|^2 + |\check{P}_{\hbar}\check{\varphi}|^2 \,dx d\ct.
    \end{equation*}
    Using the previous computation of the commutator, we obtain
    \begin{equation*}
        \int_{\R^2} |\check{B}_\hbar||\check{\varphi}|^2 m_{\hbar}^{-1}\,dx d\ct 
        \leq \check{\cQ}_{\hbar}(\check{\varphi}) + 2\hbar \int_{\R^2} \left|{\rm Im}\left(\check{P}_{\hbar}\check{\varphi} \cdot \overline{\check{\varphi}}\right)\right| |k(x)| m_{\hbar}^{-1/2} \,dx d\ct + \hbar^2 \int_{\R^2}\frac{k(x)^2}{4m_{\hbar}^2} |\check{\varphi}|^2 \,dx d\ct.
    \end{equation*}

    By the different consequences of Assumption \ref{assum:uniformtubular}, we deduce the following inequality
    \begin{multline*}
         \int_{\R^2} |\check{B}_\hbar||\check{\varphi}|^2 \,dx d\ct \leq \check{\cQ}_{\hbar}(\check{\varphi}) +  K m_{0}^{-1/2} \hbar \int_{\R^2} |\check{P}_{\hbar} \check{\varphi}|^2 \,dx d\ct \\
         + (K m_{0}^{-1/2} \hbar + \frac{1}{4}K^2 m_{0}^{-2}\hbar^2) \int_{\R^2} |\check{\varphi}|^2 \,dx d\ct.
    \end{multline*}
    Meaning we can write for some constants $c_1>0$, $C_1,>0$ independent of $h$,
    \begin{equation*}
        \check{\cQ}_{\hbar}(\check{\varphi})\geq \int_{\R^2} \left(c_1|\check{B}_\hbar|-C_1\hbar\right)|\check{\varphi}|^2 \,dx d\ct.
    \end{equation*}
    By definition we have
    \begin{equation*}
        \check{B}_\hbar(x,\ct) = \hbar^{-1}\tilde{B}(x,\hbar t).
    \end{equation*}
    Using Assumption \ref{increasingBnearGamma}, we obtain by simple integration that
    \begin{equation*}
        |\check{B}_\hbar(x,\ct)| \geq \delta_{0}|\ct|\geq \delta_{0} r.
    \end{equation*}
    Finally, we have proved
    \begin{equation*}
        \check{\cQ}_{\hbar}(\check{\varphi}) \geq (c_1\delta_{0} r - \tilde{C}_1 \hbar) \|\check{\varphi}\|^2.
    \end{equation*}
    In what follows, we choose $r_0 >0$ large enough to have
    \begin{equation}
        \label{eq:defr0}
        c_1\delta_{0} r_0 > E,
    \end{equation} 
    and the previous inequality then gives
    \begin{equation}
        \label{eq:confinementt}
        \check{\cQ}_{\hbar}(\check{\varphi}) \geq (E - \tilde{C}_1 \hbar) \|\check{\varphi}\|^2.
    \end{equation}
    
    \smallskip
    \textit{Step 2}. We now prove confinement with respect to the $x$ variable. Let $R>0$ and suppose that the support of $\check{\varphi}$ is contained in the set $\{(x,\ct)\in \R\times(-\hbar^{-1}d_0,\hbar^{-1}d_0)\,:\,|x| > R,~|\ct| < r_0\}$. 

    We now make use of the coordinates $(\fx,\ft)$ introduced by Equation \eqref{eq:coordfrak} in Section \ref{subsubsect:delta}. 
    By Equation \eqref{eq:frakquadratic} we have 
    \begin{equation*}
        \check{\cQ}_{\hbar}(\check{\varphi}) = \fQ_{\hbar}(\phi),
    \end{equation*}
    where we set $\phi = \mathfrak{U}_\hbar \varphi$ and the support of $\phi$ is now contained into the set $\{(\fx,\ft) \in \R^2\,:\,|\fx| > R,~ |\ft| \leq r_0\delta(\fx)^{1/3}\}$.
    
    Using Equations \eqref{eq:frakL} and \eqref{eq:taylorfrakA}, we can write
    \begin{multline*}
        \fQ_\hbar (\phi) = \int_{\R^2} \delta(\fx)^{2/3} |D_{\ft}\phi|^2 +\\
        \left|\mathfrak{m}_{\hbar}^{-1/2}\left(\hbar D_{\fx} - \frac{1}{2}\delta(\fx)^{1/3}\ft^2 - \hbar\left(\kappa(\fx)\delta(\fx)^{-1}\ft^3-\frac{1}{6}\delta'(\fx)\delta(\fx)^{-1}(\ft D_{\ft} + D_{\ft}\ft)\right) - \hbar^{2}\mathfrak{r}_{\hbar}\right)\mathfrak{m}_{\hbar}^{-1/2}\phi\right|^2\,d\fx d\ft,
    \end{multline*}
    We introduce $\eps > 0$, that we will determine later on. We have
    \begin{equation}
    \label{eq:mainineq}
    \begin{aligned}
        \fQ_\hbar (\phi) &\geq (1-\eps)\int_{\R^2} \delta(\fx)^{2/3} |D_{\ft}\phi|^2 + \left|\left(\mathfrak{m}_{\hbar}^{-1/2}\left(\hbar D_{\fx} - \frac{1}{2}\delta(\fx)^{1/3}\ft^2\right)\mathfrak{m}_{\hbar}^{-1/2}\right)\phi\right|^2 \,d\fx d\ft\\
        &-\eps^{-1} \hbar^2 \int_{\R^2} \left|\mathfrak{m}_{\hbar}^{-1/2}\left(\kappa(\fx)\delta(\fx)^{-1}\ft^3-\frac{1}{6}\delta'(\fx)\delta(\fx)^{-1}(\ft D_{\ft} + D_{\ft}\ft) + \hbar \mathfrak{r}_{\hbar}\right)\mathfrak{m}_{\hbar}^{-1/2}\phi\right|^2\,d\fx d\ft.
    \end{aligned}
    \end{equation}
    The second term in this difference can be bounded as follows:
    \begin{multline*}
        \int_{\R^2} \left|\mathfrak{m}_{\hbar}^{-1/2}\left(\kappa(\fx)\delta(\fx)^{-1}\ft^3-\frac{1}{6}\delta'(\fx)\delta(\fx)^{-1}(\ft D_{\ft} + D_{\ft}\ft) + \hbar \mathfrak{r}_{\hbar}\right)\mathfrak{m}_{\hbar}^{-1/2}\phi\right|^2\,d\fx d\ft \\
        \leq \int_{\R^2} \left|\left(\kappa(\fx)\delta(\fx)^{-1}\ft^3 - \frac{1}{6}\delta'(\fx)\delta(\fx)^{-1} + \hbar \mathfrak{r}_{\hbar}\right)\phi\right|^{2}\,\mathfrak{m}_{\hbar}^{-1}d\fx d\ft + \\
            \int_{\R^2}\left(\frac{\delta'(\fx)}{3\delta(\fx)}\right)^2|\mathfrak{m}_{\hbar}^{-1/2}(\ft D_{\ft}+D_{\ft}\ft)\mathfrak{m}_{\hbar}^{-1/2}\phi|^2\,d\fx d\ft \\
        \leq \left(C_{\kappa}r^3 + \hbar C_{rmd} r^4\right)^2 \int_{\R^2} |\phi|^2 \,\mathfrak{m}_{\hbar}^{-1}d\fx d\ft + (1+C_{\delta}r)^2\int_{\R^2} |\mathfrak{m}_{\hbar}^{-1/2} D_{\ft} \,\mathfrak{m}_{\hbar}^{-1/2}\phi|^2\,d\fx d\ft,
    \end{multline*}
    where the last inequality follows from Assumptions \ref{assum:delta} and \ref{assum:kappa}.

    In order to go further, we need to compare the two terms $|D_{\ft}\phi|$ and $|\mathfrak{m}_{\hbar}^{-1/2}D_{\ft}\mathfrak{m}_{\hbar}^{-1/2}\phi|$.
    To this end, observe that on the support of $\phi$ we have
    \begin{equation*}
        |\mathfrak{m}_{\hbar}| = |1 - \hbar \delta(\fx)^{-1/3}k(\fx)\ft| \geq 1 - \hbar C_{\kappa} r_0.
    \end{equation*}
    As mentioned previously, $r$ is to be fixed later and $\hbar$ can be chosen accordingly: we  suppose that we have $\hbar < \hbar_0$ where $\hbar_0$ is small enough so that, on the support of $\phi$, we can write
    \begin{equation*}
        |\mathfrak{m}_\hbar (\fx,\ft)| \geq \frac{1}{2}.
    \end{equation*}
    We deduce 
    \begin{equation*}
    \begin{aligned}
        \int_{\R^2} |\mathfrak{m}_{\hbar}^{-1/2}D_{\ft}\mathfrak{m}_{\hbar}^{-1/2}\phi|^2 \,d\fx d\ft &\leq \hbar^{2}\int_{\R^2} \frac{\mathfrak{m}_{\hbar}^{-4}}{4} |k(\fx)^{2}\delta(\fx)^{-2/3}||\phi|^2\,d\fx d\ft + \int_{\R^2} \mathfrak{m}_{\hbar}^{-2} |D_{\ft}\phi|^2 \,d\fx d\ft\\
        &\leq 4 \int_{\R^2} |D_{\ft}\phi|^2 \,d\fx d\ft + 4C_{\kappa}^2 \hbar^2 \int_{\R^2} |\phi|^2 \,d\fx d\ft.
    \end{aligned}
    \end{equation*}

    The second term in Equation \eqref{eq:mainineq} can then be bounded by
    \begin{equation*}
        4\left(C_\kappa r_0^3+ \hbar C_{rmd} r_0^4 + C_{\kappa}^2 \hbar^2\right)\int_{\R^2} |\phi|^2 \,d\fx d\ft + 4\left(1+C_\delta r_0\right)^2\int_{\R^2} |D_{\ft}\phi|^2\,d\fx d\ft.
    \end{equation*}

    By choosing $\eps = \hbar$ in Equation \eqref{eq:mainineq}, we get
    \begin{multline}
        \label{eq:ineq2}
        \mathfrak{Q}_\hbar (\phi) \geq \left((1-\hbar)\inf_{|\fx|>R}\delta(\fx)^{2/3}  -4\hbar\left(1+C_\delta r_0\right)^2\right) \fq_{\hbar}(\phi)
        \\-4\hbar \left(C_\kappa r_0^3 + \hbar C_{rmd}r_0^4 + C_{\kappa}^2 \hbar^2\right)\int_{\R^2} |\phi|^2 \,d\fx d\ft,
    \end{multline}
    where we have introduced
    \begin{equation}
        \fq_{\hbar}(\phi) = \int_{\R^2} \left(|D_{\ft}\phi|^2 + \left|\left(\mathfrak{m}_{\hbar}^{-1/2}\left(\hbar \delta(\fx)^{-1/3}D_{\fx} - \frac{1}{2}\ft^2\right)\mathfrak{m}_{\hbar}^{-1/2}\right)\phi\right|^2 \right)\,d\fx d\ft.
    \end{equation}
    So as to make the Montgomery operator and its lowest eigenvalue appear, we need to symmetrize the term $\delta(\fx)^{-1/3}D_{\fx}$: we define
    \begin{equation*}
        \Xi(\fx,D_{\fx}) = \delta(\fx)^{-1/6}D_{\fx}\delta(\fx)^{-1/6},
    \end{equation*}
    and we obtain the following new expression of the quadratic form $\fq_{\hbar}$:
    \begin{equation*}
         \fq_{\hbar}(\phi) = \int_{\R^2} \left(|D_{\ft}\phi|^2 + \left|\left(\mathfrak{m}_{\hbar}^{-1/2}\left(\hbar \Xi(\fx,D_{\fx}) - \frac{1}{2}\ft^2 - \frac{i\hbar}{6}\delta'(\fx)\delta(\fx)^{-4/3}\right)\mathfrak{m}_{\hbar}^{-1/2}\right)\phi\right|^2 \right)\,d\fx d\ft.
    \end{equation*}
    We have the lower bound
    \begin{multline*}
        \fq_{\hbar}(\phi) \geq \int_{\R^2} \left(|D_{\ft}\phi|^2 + \left|\left(\mathfrak{m}_{\hbar}^{-1/2}\left(\hbar \Xi(\fx,D_{\fx}) - \frac{1}{2}\ft^2\right)\mathfrak{m}_{\hbar}^{-1/2}\right)\phi\right|^2 \right)\,d\fx d\ft\\ 
        - \frac{2\hbar}{6} {\rm Re}\int_{\R^2} i\delta'(\fx)\delta(\fx)^{-4/3}\left(\mathfrak{m}_{\hbar}^{-1/2}\left(\hbar \Xi(\fx,D_{\fx}) - \frac{1}{2}\ft^2\right)\mathfrak{m}_{\hbar}^{-1/2}\right)\phi\,\overline{\phi}\,\mathfrak{m}_\hbar^{-1}d\fx d\ft.
    \end{multline*}
    This becomes
    \begin{equation}
        \label{eq:ineq3}
        \begin{aligned}
        \fq_{\hbar}(\phi) &\geq \int_{\R^2} \left(|D_{\ft}\phi|^2 + \left|\left(\mathfrak{m}_{\hbar}^{-1/2}\left(\hbar \Xi(\fx,D_{\fx}) - \frac{1}{2}\ft^2\right)\mathfrak{m}_{\hbar}^{-1/2}\right)\phi\right|^2 \right)\,d\fx d\ft\\ 
        &- \frac{2\hbar^2}{3} {\rm Re}\int_{\R^2} i\delta'(\fx)\delta(\fx)^{-4/3}\left(\mathfrak{m}_{\hbar}^{-1/2}\Xi(\fx,D_{\fx})\mathfrak{m}_{\hbar}^{-1/2}\right)\phi\,\overline{\phi}\,d\fx d\ft.
        \end{aligned}
    \end{equation}
    By using the identity $2{\rm Re}(\partial_{\fx}\phi\,\overline{\phi}) = \partial_{\fx}|\phi|^2$, we can bound the second term by integrating by parts:
    \begin{equation*}
        \left|\frac{2\hbar^2}{3} {\rm Re}\int_{\R^2} i\delta'(\fx)\delta(\fx)^{-4/3}\left(\mathfrak{m}_{\hbar}^{-1/2}\Xi(\fx,D_{\fx})\mathfrak{m}_{\hbar}^{-1/2}\right)\phi\,\overline{\phi}\,d\fx d\ft\right| \leq \tilde{C}_1\hbar^{2} \int_{\R^2}|\phi|^2\,d\fx d\ft,
    \end{equation*}
    for some constant $\tilde{C}_1>0$.

    Meanwhile, the first term in Equation \eqref{eq:ineq3} can be bounded using functional calculus in order to make apparent the relation with the Montgomery operators. Indeed, $\Xi(\fx,D_{\fx})$ is a self-adjoint operator, so it can be diagonalized. 
    The only thing left in the way is once again the weight $\mathfrak{m}_{\hbar}$. As we did before, up to terms of order $O(\hbar^2 \Vert \phi\Vert)$, we can move around this weight so as to obtain:
    \begin{equation*}
        \fq_{\hbar}(\phi) \geq  \int_{\R^2} |D_{\ft}\phi|^2 + \left|\left(\hbar \Xi(\fx,D_{\fx}) - \frac{1}{2}\ft^2\right)\phi\right|^2 \,d\fx d\ft - \tilde{C}_2\hbar^{2}\int_{\R^2} |\phi|^2\,d\fx d\ft
    \end{equation*}
    for some other positive constant $\tilde{C}_2$.
    Using functional calculus and the min-max theorem we obtain,
    \begin{equation}
        \label{eq:ineq4}
        \begin{aligned}
        \fq_{\hbar}(\phi) &\geq \int_{\R^2} \mu_{1}(\zeta)|\phi|^2\,d\zeta d\ft - \tilde{C}_2\hbar^{2}\int_{\R^2} |\phi|^2\,d\fx d\ft\\
        &\geq (\tilde - \tilde{C}_2\hbar^2) \int_{\R^2} |\phi|^2\,d\fx d\ft.
        \end{aligned}
    \end{equation}
    We can now combine Equations \eqref{eq:ineq2} and \eqref{eq:ineq4} to obtain
    \begin{equation*}
        \mathfrak{Q}_{\hbar}(\phi) \geq \left((1-\hbar)\inf_{|\fx|>R}\delta(\fx)^{2/3}\tilde{\mu}_{1,c} - 4\hbar(1+C_{\delta}r_0)^2 - 4\hbar (C_\kappa r_0^3+ \hbar C_{rmd}r_0^4 + \tilde{C}_2 \hbar + C_{\kappa}^2 \hbar^2)\right) \|\phi\|_{L^2}^2.
    \end{equation*}
    Let's take $R_0$ big enough so that $\delta(\fx) > \delta_{*}$ for $|\fx| > R_0$: by the previous inequality, there exists $\hbar_{0}>0$ such that for $\hbar \in (0,\hbar_{0})$, we have
    \begin{equation}
        \label{eq:confinementfx}
        \mathfrak{Q}_{\hbar}(\phi) \geq (E - \tilde{C}_3 \hbar)  \int_{\R^2} |\phi|^2\,d\fx d\ft,
    \end{equation}
    for some positive constant $\tilde{C}_3$.
    
    \medskip
    \textit{Gluing the lower bounds}.
    Using the inequalities \eqref{eq:awayfromGamma},\eqref{eq:confinementt} and \eqref{eq:confinementfx}, we deduce that there exist $r_0 >0$, $R_0 >0$ and $h_0>0$ such that for all test function
    $\psi\in C_{c}^{\infty}(\R^2)$ supported either in $\R^2\setminus \Omega_0$, in $\Phi(\{(x,t)\in\R\times(-d_0,d_0)\,:\, |t|>r_0\hbar\})$, or in $\Phi(\{(x,t)\in\R\times(-d_0,d_0)\,:\,|x|> R_0,\ |t|\leq r_0\hbar\})$,
    and for $h\in (0,h_0)$, we have
    \begin{equation*}
        \cQ_{h}(\psi) \geq (E-\tilde{C}h^{1/3})h^{4/3}\|\psi\|_{L^2(\R^2)}^2,
    \end{equation*}
    for some positive constant $\tilde{C}$.

    We concentrate first on gluing the two lower bounds associated to the neighborhoods of $\Gamma$. We introduce a partition of unity with respect to $\ct$ such that 
    \begin{equation}
        \label{eq:partitionIMS}
        \chi_{1,r_0}^2 + \chi_{2,r_0}^2 = 1,\ \chi_{1,r_0} = \begin{cases}1 \mbox{ for } |\ct|\leq\frac{r_0}{2},\\ 0\mbox{ for }|\ct|\geq r_0,\end{cases}\quad\mbox{and}\quad \left(\chi_{1,r_0}'\right)^2 +  \left(\chi_{2,r_0}'\right)^2 \leq C r_0^{-2}.
    \end{equation}
    Let $\varphi \in C_{c}^{\infty}(\R^2)$ supported in $\Omega_0$, and still denoting by $\check{\varphi}$ the corresponding test function in the coordinates $(x,\ct)$, we suppose that $\check{\varphi}$ is supported away from $[-R_0,R_0]\times[-r_0,r_0]$.
    The IMS formula allows us to write:
    \begin{equation}
        \label{eq:boundcheckQ}
        h^{-4/3}\cQ_{h}(\varphi) = \check{\cQ}_{\hbar}(\check{\varphi}) \geq \check{\cQ}_{\hbar}(\chi_{1,r_0}\check{\varphi}) + \check{\cQ}_{\hbar}(\chi_{2,r_0}\check{\varphi}) - C r_{0}^{-2}\|\varphi\|^2.
    \end{equation}
    By support condition, the next estimate holds:
    \begin{equation*}
        \cQ_{h}(\varphi) \geq h^{4/3} (E-\tilde{C}h^{1/3} - Cr_{0}^{-2})\|\varphi\|^2.
    \end{equation*}
    Therefore, up to choosing $r_0$ larger and possibly $h_0$ smaller, we obtain the desired lower bound for any $E\in(0,E_0)$.
    Gluing this lower bound with the one away from the tubular neighborhood $\Omega_0$ is done in a similar fashion, using a partition of unity with respect to the distance at $\Gamma$.
\end{proof}

From the previous analysis it is possible to obtain a (very) rough Weyl law for the magnetic Laplacian in the energy window $(-\infty,Eh^{4/3}]$. A more accurate estimate 
of the eigenvalue counting function will be given in Section \ref{sect:dimreduction}.

\begin{corollary}
    \label{cor:roughweyllaw1}
    Let $E\in(-\infty,E_0)$. There exists $C>0$ such that the number of eigenvalues (counted with multiplicities) of $\magSch$ in the energy window $(-\infty,Eh^{4/3}]$, denoted by 
    $N(\magSch,Eh^{4/3})$, is bounded as follows
    \begin{equation*}
        N(\magSch,Eh^{4/3}) \leq C h^{-2}.
    \end{equation*}
\end{corollary}

\begin{proof}
    Similarly as in Equation \eqref{eq:partitionIMS}, we introduce a partition of unity with respect to $t$ such that 
    \begin{equation*}
        \tilde{\chi}_{1,d_0}^2 + \tilde{\chi}_{2,d_0}^2 = 1,\ \tilde{\chi}_{1,d_0} = \begin{cases}1 \mbox{ for } |t|\leq\frac{d_0}{2},\\ 0\mbox{ for }|t|\geq d_0,\end{cases}\quad\mbox{and}\quad \left(\tilde{\chi}_{1,d_0}'\right)^2 +  \left(\tilde{\chi}_{2,d_0}'\right)^2 \leq \tilde{C},
    \end{equation*}
    for some positive $\tilde{C}>0$ and we consider the functions $z\in\R^{2} \mapsto \tilde{\chi}_{j,d_0}({\rm dist}(z,\Gamma))$, $j=1,2$, that we still denoted by $\tilde{\chi}_{j,d_0}$. This defines a partition of unity on $\R^2$, and by the IMS formula 
    we have, for $\varphi\in \dom(\magSch)$,
    \begin{equation*}
        \cQ_h(\varphi) \geq \cQ_h(\tilde{\chi}_{1,d_0} \varphi) + \cQ_h(\tilde{\chi}_{2,d_0} \varphi) - \tilde{C}h^2\|\varphi\|^2.
    \end{equation*}
    Coupling this inequality with Equation \eqref{eq:boundcheckQ}, we obtain 
    \begin{equation}
        \label{eq:boundbelowIMS}
        \cQ_h(\varphi) \geq \cQ_h(\tilde{\chi}_{2,d_0} \varphi) + h^{4/3}\check{\cQ}_{\hbar}(\chi_{2,r_0}\check{\psi})+ h^{4/3}\check{\cQ}_{\hbar}(\chi_{2,r_0}\check{\psi}) - h^{4/3}(\tilde{C}h^{2/3}+C r_{0}^{-2})\|\varphi\|^2,
    \end{equation}
    where $\psi = \tilde{\chi}_{1,d_0}\varphi$ and the change of coordinates $(x,\ct)$ thus makes sense.
    
    Denoting by $I_{r_0}$ the open set $(-\infty,-r_0/2)\cup(r_0/2,+\infty)$, we consider the bounded linear isomorphism 
    \begin{equation*}
        D: \varphi\in L^2(\R^2) \mapsto (\tilde{\chi}_{2,d_0} \varphi,\chi_{1,r_0}\check{\psi},\chi_{2,r_0}\check{\psi})\in L^2(\R^2\setminus \Omega_{d_0/2})\oplus L^2(\R\times I_{r_0})\oplus L^2(\R\times(-r_0,r_0)),
    \end{equation*}
    whose inverse is given by 
    \begin{equation*}
        D^{-1}(\varphi,\check{\psi}_1,\check{\psi}_2) = \tilde{\chi}_{2,d_0} \varphi + \tilde{\chi}_{1,d_0}\left(\chi_{1,r_0}\psi_1+ \chi_{2,r_0}\psi_2\right),
    \end{equation*}
    where $\psi_1$ and $\psi_2$ are the norm-preserving pull-back of $\check{\psi}_1$ and $\check{\psi}_1$ by the dilated tubular coordinates in $\Omega_{d_0}$.

    The min-max theorem and the inequality Equation \eqref{eq:boundbelowIMS} then implies that, for $r_0$ chosen large enough and $h$ small enough, the counting function $N(\magSch,Eh^{4/3})$ is bounded by the number of eigenvalues 
    in the energy window $(-\infty,Eh^{4/3}]$ of the operator $\magSch \oplus h^{4/3}\check{\mathcal{L}}_\hbar \oplus h^{4/3}\check{\mathcal{L}}_\hbar$ with domain 
    \begin{equation*}
        \{(u_1,u_2,u_3)\in L^2(\R^2\setminus \Omega_{d_0/2})\oplus L^2(\R\times I_{r_0})\oplus L^2(\R\times(-r_0,r_0))\,:\, \magSch(D^{-1}(u_1,u_2,u_3))\in L^2(\R^2)\}.
    \end{equation*}
    The eigenvalues of the operator $\magSch \oplus h^{4/3}\check{\mathcal{L}}_\hbar \oplus h^{4/3}\check{\mathcal{L}}_\hbar$ are given by the sum of each of these operators on their respective domains.
    From Equations \eqref{eq:awayfromGamma} and \eqref{eq:confinementt}, it is straightforward to deduce that the spectrum of $\magSch$ on $L^2(\R^2\setminus \Omega_{d_0})$ and the one of $h^{4/3}\check{\mathcal{L}}_\hbar$ 
    on $L^2(\R\times I_{r_0})$ are bounded by below by $h b_0$ and $h^{4/3}E'$ respectively, where $E'\in(E,E_0)$ and if $r_0$ is chosen large enough. Thus, the only eigenvalues in the energy window 
    $(-\infty,Eh^{4/3}]$ comes from the operator $h^{4/3}\check{\mathcal{L}}_\hbar$ on the domain $L^{2}(\R\times (-r_0,r_0))$.

    The analysis of this operator has been already performed in the proof of Theorem \ref{thm:spectrediscret}. Up to introducing a new partition of unity for the $x$ variable,
    \begin{equation*}
        \check{\chi}_{1,R_0}^2 + \check{\chi}_{2,R_0}^2 = 1,\ \check{\chi}_{1,d_0} = \begin{cases}1 \mbox{ for } |x|\leq\frac{R_0}{2},\\ 0\mbox{ for }|x|\geq R_0,\end{cases}\quad\mbox{and}\quad \left(\check{\chi}_{1,R_0}'\right)^2 +  \left(\check{\chi}_{2,R_0}'\right)^2 \leq \check{C},
    \end{equation*}
    we can write the corresponding IMS formula and use Equation \eqref{eq:confinementfx}, for $R_0$ chosen large enough, we can bound $N(\magSch,Eh^{4/3})$ by the number of eigenvalues in $(-\infty,Eh^{4/3}]$ of the operator 
    $h^{4/3}\check{\mathcal{L}}_\hbar$ for the domain 
    \begin{equation*}
        \{u\in L^{2}([-R_0,R_0]\times[-r_0,r_0])\,:\,\check{\mathcal{L}}_\hbar u \in L^{2}([-R_0,R_0]\times[-r_0,r_0])\}.
    \end{equation*}
    This can in turn be very roughly bounded by the counting function of the operator $\magSch$ restricted to the bounded domain 
    \begin{equation*}
        \Omega_{R_0,d_0/2} = \Phi^{-1}([-R_0,R_0]\times [-d_0/2,d_0/2]),
    \end{equation*}
    with Dirichlet condition, i.e. by the number of eigenvalues of the magnetic Laplacian $\magSch$ under the energy $Eh^{4/3}$, or even, again trying to find a rough upper bound, under the energy level 
    $Eh$, when restricted to functions supported in $\Omega_{R_0,d_0/2}$. Then, by standard Weyl laws for magnetic Laplacian on bounded domains, we obtain the desired result.
\end{proof}

\section{Microlocalization on a compact region}
\label{sect:microlocalization}

\subsection{Microlocalization properties}

As we have established the existence of a discrete spectrum in the energy window $(-\infty,Eh^{4/3}]$ (for any $E< E_0$), we are interested in the localization properties of the corresponding eigenfunctions.

\subsubsection{Localization close to the zero locus}
The first result, already partially recalled in Section \ref{sect:intro}, concerns their localization close to the zero locus $\Gamma$ of the magnetic field and is expressed by the following Agmon-type estimate (see \cite[Proposition 3.4]{DR}).

\begin{proposition}
    \label{prop:localizationGamma}
    Let $E \in (-\infty,E_0)$. There exist $C>0$, $\alpha >0$ and $h_0 > 0$ such that, for $h\in(0,h_0 ]$ and for any eigenpair $(\mu_h, \psi_h)$ of $\magSch$ satisfying $\mu_h \leq Eh^{4/3}$, we have
    \begin{equation*}
        \begin{aligned}
            \int_{\R^2} e^{2\alpha h^{-1/3}{\rm dist}(z,\Gamma)} |\psi_h (z)|^2 \,dz &\leq C \Vert \psi_h \Vert^2, \\
            \cQ_{h}\left(e^{\alpha h^{-1/3}{\rm dist}(\cdot,\Gamma)} \psi_h \right) &\leq C h^{4/3}\Vert \psi_h \Vert^2.
        \end{aligned}
    \end{equation*}
\end{proposition}

This estimate allows us to reduce the spectral study in the energy window $(-\infty, Eh^{4/3}]$ of our magnetic Laplacian in a tubular neighborhood of $\Gamma$ of distance $\hbar^{1-\eta}$ for any $\eta \in (0, 1)$. 
For this reason, we introduce the open set
\begin{equation*}
    \Omega_{\hbar,\eta}  =\{z \in\R^2\,:\,{\rm dist}(z,\Gamma) < \hbar^{1-\eta}\},
\end{equation*}
where in particular, for $\hbar$ small enough, the changes of coordinates $(x,t)$ and $(x,\ct)$ make sense. We also introduce a cutoff function $\theta \in C_{0}^\infty(\R,[0,1])$ such that 
$\theta = 1$ in a neighborhood of 0. Then, for any eigenpair $(\mu_h, \psi_h)$ of $\magSch$ satisfying $\mu_h \leq Eh^{4/3}$, we can consider $\theta(\hbar^{\eta-1}{\rm dist}(\cdot,\Gamma))\psi_h$ and its 
transformation into the coordinates $(x,\ct)$, denoted by $\check{\psi}_\hbar$, as in Equation \eqref{eq:checkvarphi}:
\begin{equation}
    \label{eq:checkpsi}
    \check{\psi}_\hbar = \hbar^{1/2}(\theta(\hbar^{\eta-1}{\rm dist}(\cdot,\Gamma))\psi_h)(\Phi(x,\hbar t)).
\end{equation}
Observe that we obtain a quasi-mode of $\check{\mathcal{L}}_\hbar$ for the eignevalue $\check{\mu}_\hbar = h^{-4/3}\mu_h$:
\begin{equation}
    \label{eq:quasimode}
    \check{\mathcal{L}}_\hbar \check{\psi}_\hbar = \check{\mu}_\hbar\check{\psi}_\hbar + \O(\hbar^{\infty}).
\end{equation}

\subsubsection{Localization on the curve $\Gamma$}
We are now looking to establish the localization of low-energy eigenfunctions with respect to the space variable $x$ parametrizing $\Gamma$.

As the proof of Theorem \ref{thm:spectrediscret} has shown, the transversal derivative $\delta$ and is behavior at infinity is the reason for the existence of a confinement. 
It is thus natural to consider the Lithner-Agmon distance associated with $\delta$. More precisely, for a level $E\in (-\infty,E_0 h^{4/3}]$ of energy prescribed, we consider 
the following metric on $\Gamma$:
\begin{equation*}
    (\delta(x)^{2/3}\tilde{\mu}_{1,c} - E)_{+}^{1/2} \,dx^2,\ x\in\R,
\end{equation*}
where $a_{+} = {\rm max}(a,0)$ and $dx$ is the curvilineal one-form on the curve $\Gamma$. In what follows, the LA distance $d_{LA}(x,y)$ between two points $x$ and $y$ on the 
$\Gamma$ will be the length for the previous metric of the segment of $\Gamma$ joining the two points.

The following lemma is a necessary first Agmon-type estimate to treat the transverse variable $\ct$.

\begin{lemma}
    \label{lem:weightedAgmon}
    Let $E \in(-\infty,E_0)$, $\eps\in(0,1)$, $M>0$ and $\bz:\R\rightarrow\R$ a Lipschitzian function. There exist $h_0>0$ and $C>0$ such that, for all eigenpairs $(\mu_h,\psi_h)$ of $\magSch$ satisfying $\mu_h\leq E h^{4/3}$, if 
    we consider the quasi-mode $\check{\psi}_{\hbar}$ as in Equation \eqref{eq:checkpsi}, we have for all $p\geq 1$ and $h\leq h_0$:
    \begin{equation*}
        \begin{aligned}
            \|e^{M |\ct| + (1-\eps)\hbar^{-1}\chi_p \bz}\check{\psi}_{\hbar}\| &\leq C\|e^{(1-\eps)\hbar^{-1}\chi_p \bz}\check{\psi}_{\hbar}\|,\\
            \check{\cQ}_{\hbar}\left(e^{M |\ct| + (1-\eps)\hbar^{-1}\chi_p \bz}\check{\psi}_{\hbar}\right)&\leq C\|e^{(1-\eps)\hbar^{-1}\chi_p \bz}\check{\psi}_{\hbar}\|^2,
        \end{aligned}
    \end{equation*} 
    where $\chi_{p}(s) = \chi(p^{-1}s)$, with $0\leq \chi \leq 1$ a smooth cutoff function supported near 0.$\footnote{Dans le papier avec Dombrowski, les estimées sont sous-optimales parce qu'on fait une estimation de la forme quadratique pour v grand, donc avec un confinement meilleur!}$
\end{lemma}

\begin{proof}
    By Equation \eqref{eq:quasimode}, we know that $\check{\psi}_{\hbar}$ is a quasi-mode for the operator 
    $\check{\mathcal{L}}_\hbar$. Considering the Lipschitzian function $L_p(x,\ct) = M |\ct| + (1-\eps) \hbar^{-1}\chi_p \bz(x)$, we have the following Agmon formula
    \begin{equation*}
        \check{\cQ}_{\hbar}(e^{L_p} \check{\psi}_{\hbar}) = \check{\mu}_{\hbar}\|e^{L_p} \check{\psi}_{\hbar}\|^2 + \|\partial_{\ct}L e^{L_p} \check{\psi}_{\hbar}\|^2 + \|\hbar D_{x}L e^{L_p} \check{\psi}_{\hbar}\|^2 + \O(\hbar^\infty)\|e^{L_p} \check{\psi}_{\hbar}\|^2,
    \end{equation*}
    from which we get the following inequality 
    \begin{equation*}
        \check{\cQ}_{\hbar}(e^{L_p} \check{\psi}_{\hbar}) \leq (E + M^2+\O(\hbar^\infty))\|e^{L_p} \check{\psi}_{\hbar}\|^2 + (1-\eps)^2\|\bz' e^{L_p} \check{\psi}_{\hbar}\|^2 + (1-\eps)^2\|\bz \chi_{p}' e^{L_p} \check{\psi}_{\hbar}\|^2.
    \end{equation*}
    Since $\bz$ is a Lipschitzian function, there exists $K\geq 0$ such that for all $x\in\R$, we have 
    \begin{equation*}
        |\bz(x)| \leq |\bz(0)| + K|x|.
    \end{equation*}
    Therefore, for all $p\geq 1$, $\eps > 0$ and $M > 0$, we can write 
    \begin{equation*}
        \check{\cQ}_{\hbar}(e^{L_p} \check{\psi}_{\hbar}) \leq (E + M^2+ K(1-\eps)^2 + \O(\hbar^\infty))\|e^{L_p} \check{\psi}_{\hbar}\|^2 + (1-\eps)^2\|\bz \chi_{p}' e^{L_p} \check{\psi}_{\hbar}\|^2.
    \end{equation*}
    We introduce a partition of unity 
    \begin{equation*}
        \chi_{1,r}(\ct)^2 + \chi_{2,r}(\ct)^2 = 1,
    \end{equation*}
    where $\chi_{2,r}$ is supported in $\{|\ct|\geq r\}$. We assume we have chosen this partition in such way that there exists $C>0$ such that for all $r>0$
    \begin{equation*}
        \chi_{1,r}'^{2}+\chi_{2,r}'^{2} \leq Cr^{-2}.
    \end{equation*}
    The IMS formula implies that 
    \begin{multline*}
        \check{\cQ}_{\hbar}(\chi_{1,r}e^{L_p} \check{\psi}_{\hbar})+\check{\cQ}_{\hbar}(\chi_{2,r}e^{L_p} \check{\psi}_{\hbar})-Cr^{-2}\|e^{L_p} \check{\psi}_{\hbar}\|^2 \leq (E + M^2+ K(1-\eps)^2 + \O(\hbar^\infty))\|e^{L_p} \check{\psi}_{\hbar}\|^2 \\+ (1-\eps)^2\|\bz \chi_{p}' e^{L_p} \check{\psi}_{\hbar}\|^2.
    \end{multline*}
    By the inequality \eqref{eq:confinementt} in the proof of Theorem \ref{thm:spectrediscret}, we can choose $r_0$ sufficiently large and $h_0$ small enough such that we have, for $h\in(0,h_0)$ and $p\geq 1$ 
    \begin{equation*}
        \check{\cQ}_{\hbar}(\chi_{2,r_0}e^{L_p} \check{\psi}_{\hbar}) \geq \delta_{c} r_0 \|\chi_{2,R_0}e^{L_p} \check{\psi}_{\hbar}\|^2,
    \end{equation*}
    and 
    \begin{equation*}
        E + M^2+ K(1-\eps)^2 + C r_{0}^{-2} \leq \frac{1}{2}\delta_{c}r_0.
    \end{equation*}
    For these choices, we find that for all $h\in(0,h_0)$ and $p\geq 1$, we have 
    \begin{equation*}
        \frac{1}{2}\delta_{c}r_0\|\chi_{2,R_0}e^{L_p} \check{\psi}_{\hbar}\|^2 \leq \frac{1}{2}\delta_{c}r_0e^{Mr_0}\|e^{(1-\eps)\hbar^{-1}\chi_{p}\bz}\check{\psi}_{\hbar}\|^2 + (1-\eps)^2\|\bz \chi_{p}' e^{L_p} \check{\psi}_{\hbar}\|^2,
    \end{equation*}
    and letting $p$ go to infinity, the last term disappears as we have $\chi_{p}'(\cdot) = p^{-1}\chi'(p^{-1}\cdot)$, and the conlusion follows.
\end{proof}

We are now able to prove the desired localization with respect to the variable $x$.

\begin{proposition}
    \label{prop:AgmonGamma}
    Let $E \in(-\infty,E_0)$. We consider the compact set
    \begin{equation*}
        K_E = \{x \in \Gamma \,:\, \delta(x)^{2/3} \tilde{\mu}_{1,c} \leq E\},
    \end{equation*}
    and we introduce the Lipschitzian function $\bz(x) = d_{LA}(x,K)$ for $x\in\Gamma$.
    Then, for every $\eps\in(0,1)$, there exist $C_\eps>0$ and $h_0 > 0$ such that, for $h\in(0,h_0]$ and for all eigenpair $(\mu_h, \psi_h)$ of $\magSch$ satisfying $\mu_h \leq Eh^{4/3}$, we have
    \begin{equation}
        \label{eq:AgmonGamma}
        \begin{aligned}
            \Vert e^{(1-\eps) \hbar^{-1} \bz} \check{\psi}_{\hbar}\Vert &\leq C_\eps e^{\eps/\hbar}\Vert \check{\psi}_{\hbar}\Vert,\\
            \check{\cQ}_{\hbar}(e^{(1-\eps) \hbar^{-1} \bz}\check{\psi}_{\hbar}) &\leq C_\eps e^{\eps/\hbar}\Vert \check{\psi}_{\hbar}\Vert^2.
        \end{aligned}
    \end{equation}
\end{proposition}

\begin{proof}
Recall that by Equation \eqref{eq:quasimode}, $\check{\psi}_\hbar$ is a quasi-mode of $\check{\mathcal{L}}_\hbar$ for the eignevalue $\check{\mu}_\hbar = h^{-4/3}\mu_h$:
\begin{equation*}
    \check{\mathcal{L}}_\hbar \check{\psi}_\hbar = \check{\mu}_\hbar \check{\psi}_\hbar + \O(\hbar^{\infty}).
\end{equation*}
We start from the following Agmon-type estimate on $\check{\mathcal{L}}_\hbar$:
\begin{multline*}
     \check{\cQ}_{\hbar}(e^{(1-\eps) \hbar^{-1}\chi_{p} \bz} \check{\psi}_\hbar) \leq (\check{\mu}_\hbar+\O(\hbar^\infty))\Vert e^{(1-\eps) \hbar^{-1}\chi_{p} \bz} \check{\psi}_\hbar\Vert^{2} + (1-\eps)^{2}\Vert \bz'e^{(1-\eps) \hbar^{-1}\chi_{p} \bz} \check{\psi}_\hbar\Vert^{2} 
     \\+ (1-\eps)^{2}\Vert \chi_{p}'\bz e^{(1-\eps) \hbar^{-1} \chi_{p}\bz} \check{\psi}_\hbar\Vert^{2}.
\end{multline*}
We introduce the notation $\psi_{\hbar}^{\rm wg} = \mathfrak{U}_\hbar (e^{(1-\eps) \hbar^{-1} \chi_{p}\bz} \check{\psi}_\hbar)$, where the norm-preserving operator $\mathfrak{U}_\hbar$ has been introduced in 
Equation \eqref{eq:reductionfrak} and corresponds to the change of variable $(\fx,\ft)$ given by Equation \eqref{eq:coordfrak}. We reformulate the previous equation as follows
\begin{equation}
    \label{eq:standardagmon}
    \mathfrak{Q}_{\hbar}(\psi_{\hbar}^{\rm wg}) \leq \check{\mu}_\hbar \Vert \psi_{\hbar}^{\rm wg}\Vert^2 + (1-\eps)^{2}\Vert \bz'\psi_{\hbar}^{\rm wg} \Vert^{2} + (1-\eps)^{2}\Vert \chi_{p}'\bz \psi_{\hbar}^{\rm wg} \Vert^{2} + \O(\hbar^\infty)\Vert \psi_{\hbar}^{\rm wg}\Vert^2.
\end{equation}

By Equation \eqref{eq:mainineq} that we established in the proof of Theorem \ref{thm:spectrediscret}, we found 
\begin{equation*}
    \begin{aligned}
        \mathfrak{Q}_\hbar (\psi_{\hbar}^{\rm wg}) &\geq (1-\hbar)\int_{\R^2} \delta(\fx)^{2/3} |D_{\ft}\psi_{\hbar}^{\rm wg}|^2 + \left|\left(\mathfrak{m}_{\hbar}^{-1/2}\left(\hbar D_{\fx} - \frac{1}{2}\delta(\fx)^{1/3}\ft^2\right)\mathfrak{m}_{\hbar}^{-1/2}\right)\psi_{\hbar}^{\rm wg}\right|^2 \,d\fx d\ft\\
        &-\hbar \int_{\R^2} \left|\mathfrak{m}_{\hbar}^{-1/2}\left(\kappa(\fx)\delta(\fx)^{-1}\ft^3-\frac{1}{6}\delta'(\fx)\delta(\fx)^{-1}(\ft D_{\ft} + D_{\ft}\ft) + \hbar \mathfrak{r}_{\hbar}\right)\mathfrak{m}_{\hbar}^{-1/2}\psi_{\hbar}^{\rm wg}\right|^2\,d\fx d\ft.
    \end{aligned}
\end{equation*}
The analysis in the proof of Theorem \ref{thm:spectrediscret} has also shown that the last term of this inequality can be bounded as follow
\begin{multline*}
    \int_{\R^2} \left|\mathfrak{m}_{\hbar}^{-1/2}\left(\kappa(\fx)\delta(\fx)^{-1}\ft^3-\frac{1}{6}\delta'(\fx)\delta(\fx)^{-1}(\ft D_{\ft} + D_{\ft}\ft) + \hbar \mathfrak{r}_{\hbar}\right)\mathfrak{m}_{\hbar}^{-1/2}\psi_{\hbar}^{\rm wg}\right|^2\,dx d\ft \leq \\C \|\ft D_{\ft}\psi_{\hbar}^{\rm wg}\|^2 + C\|\psi_{\hbar}^{\rm wg}\|^2,
\end{multline*}
for some positive constant $C$.
As we have
\begin{equation*}
    \ft D_{\ft} = \ct D_{\ct},
\end{equation*}
we deduce from Lemma \ref{lem:weightedAgmon} that we have
\begin{equation*}
    \|\ft D_{\ft}\psi_{\hbar}^{\rm wg}\| \leq C\|\psi_{\hbar}^{\rm wg}\|.
\end{equation*}
We infer that
\begin{multline*}
    \mathfrak{Q}_\hbar (\psi_{\hbar}^{\rm wg}) \geq (1-\hbar)\int_{\R^2} \delta(\fx)^{2/3} |D_{\ft}\psi_{\hbar}^{\rm wg}|^2 + \left|\left(\mathfrak{m}_{\hbar}^{-1/2}\left(\hbar D_{\fx} - \frac{1}{2}\delta(\fx)^{1/3}\ft^2\right)\mathfrak{m}_{\hbar}^{-1/2}\right)\psi_{\hbar}^{\rm wg}\right|^2 \,d\fx d\ft
    \\-2C \hbar \|\psi_{\hbar}^{\rm wg}\|^2.
\end{multline*}
Now observe that we can write 
\begin{equation*}
    \begin{aligned}
    &\int_{\R^2} \delta(\fx)^{2/3} |D_{\ft}\psi_{\hbar}^{\rm wg}|^2 + \left|\left(\mathfrak{m}_{\hbar}^{-1/2}\left(\hbar D_{\fx} - \frac{1}{2}\delta(\fx)^{1/3}\ft^2\right)\mathfrak{m}_{\hbar}^{-1/2}\right)\psi_{\hbar}^{\rm wg}\right|^2 \,d\fx d\ft\\
    &\geq \int_{\R^2} |D_{\ft}\varphi_{\hbar}^{\rm wg}|^2 + \left|\left(\hbar D_{\fx}\delta(\fx)^{-1/3} - \frac{1}{2}\ft^2\right)\varphi_{\hbar}^{\rm wg}\right|^2\,d\fx d\ft - c\hbar^2\|\psi_{\hbar}^{\rm wg}\|^2\\
    &= \int_{\R^2} |D_{\ft}\varphi_{\hbar}^{\rm wg}|^2 + \left|\left(\hbar \Xi(\fx,D_{\fx}) - \frac{1}{2}\ft^2 + \frac{i\hbar}{6}\delta'(\fx)\delta(\fx)^{-4/3}\right)\varphi_{\hbar}^{\rm wg}\right|^2\,d\fx d\ft - c\hbar^2\|\psi_{\hbar}^{\rm wg}\|^2,
    \end{aligned}
\end{equation*}
where we have introduced
\begin{equation*}
    \varphi_{\hbar}^{\rm wg}(\fx,\ft) = \delta(\fx)^{1/3}\psi_{\hbar}^{\rm wg}(\fx,\ft),
\end{equation*}
and for some positive constant $c$.
We are reduced to a similar analysis as for the Equation \eqref{eq:ineq2}: indeed, for some positive constant $\tilde{c}$, we can write 
\begin{equation*}
    \mathfrak{Q}_\hbar(\psi_{\hbar}^{\rm wg}) \geq (1-\hbar)\check{\fq}_{\hbar}(\varphi_{\hbar}^{\rm wg}) - \tilde{c}\hbar \|\psi_{\hbar}^{\rm wg}\|^2.
\end{equation*}
Using here again the functional calculus associated to the self-adjoint operator $\Xi(\fx,D_{\fx})$, we obtain
\begin{equation*}
    \mathfrak{Q}_\hbar(\psi_{\hbar}^{\rm wg}) \geq (1-\hbar) \tilde{\mu}_{1,c} \|\delta^{1/3}\psi_{\hbar}^{\rm wg}\|^2 -\tilde{c}\hbar \|\psi_{\hbar}^{\rm wg}\|^2.
\end{equation*}

Coming back to Equation \eqref{eq:standardagmon}, we are now working with the following inequality:
\begin{multline*}
    \tilde{\mu}_{1,c}\| \delta^{1/3} \psi_{\hbar}^{\rm wg} \|^2 - \tilde{c}\hbar \|\psi_{\hbar}^{\rm wg}\|^2 \leq \check{\mu}_\hbar \|\psi_{\hbar}^{\rm wg}\|^2 + (1-\eps)^{2}\Vert \bz'e^{(1-\eps) \hbar^{-1}\chi_{p} \bz} \psi_{h}\Vert^{2} \\
    + (1-\eps)^2\Vert \chi_{p}'\bz e^{(1-\eps) \hbar^{-1} \chi_{p}\bz} \psi_{h}\Vert^{2} + \O(\hbar^\infty)\Vert \psi_{\hbar}^{\rm wg}\Vert^2.
\end{multline*}

What follows is a classical argument to obtain the desired Agmon estimates \eqref{eq:AgmonGamma}. We can rewrite the last inequality as 
\begin{equation*}
    \int_{\R^2} \left(\delta(\fx)^{2/3} \tilde{\mu}_{1,c} - E - (1-\eps)^2(\bz')^2(\fx)-\tilde{c}\hbar\right)|\psi_{\hbar}^{\rm wg}|^2\,d\fx d\ft \leq(1-\eps)^2 \Vert \chi_{p}'\bz e^{(1-\eps) \hbar^{-1} \chi_{p}\bz} \psi_{h}\Vert^{2}.
\end{equation*}
We introduce $K_{E,\eps} := \{x \in \Gamma \,:\, \delta(x)^{2/3} \tilde{\mu}_{1,c} -E\leq \eps\}$. Then, outside $K_{E,\eps}$, we have
\begin{equation*}
    \begin{aligned}
    \delta(x)^{2/3} \tilde{\mu}_{1,c} - E - (1-\eps)^2(\bz')^2 &\geq \delta(x)^{2/3} \tilde{\mu}_{1,c} - E - (1-\eps)^2(\delta(x)^{2/3}\tilde{\mu}_{1,c} - E)_{+} \\
    &\geq (1-(1-\eps)^2)(\delta(x)^{2/3} \tilde{\mu}_{1,c} - E) \\
    &\geq (1-(1-\eps)^2)\eps \geq \frac{1}{2}\eps^2,
    \end{aligned}
\end{equation*}
for $\eps$ small enough. It follows that we have 
\begin{multline*}
    \frac{1}{2}\eps^2 \|\psi_{\hbar}^{\rm wg}\|_{L^2(\R^2\setminus K_{E,\eps})}^2 \leq -\int_{K_{E,\eps}} \left(\delta(x)^{1/3} \tilde{\mu}_{1,c} - E - (1-\eps)^2(\bz')^2(u)-\tilde{c}\hbar\right)|\psi_{\hbar}^{\rm wg}|^2\,dx d\ft \\
    +(1-\eps)^2\Vert \chi_{p}'\bz e^{(1-\eps) \hbar^{-1} \chi_{p}\bz} \psi_{h}\Vert^{2}.
\end{multline*}
As $K_{E,\eps}$ is compact and as we have $\bz \leq \eps$ on $K_{E,\eps}$, we obtain for some constant $\tilde{C}_\eps >0$:
\begin{equation*}
    \|\psi_{\hbar}^{\rm wg}\|^2 \leq \tilde{C}_\eps e^{\eps/\hbar}\|\phi_{\hbar}\|^2 + 2\eps^{-2}(1-\eps)^2\Vert \chi_{p}'\bz e^{(1-\eps) \hbar^{-1} \chi_{p}\bz} \psi_{h}\Vert^{2},
\end{equation*}
and letting $p$ go to infinity, as $|\chi_{p}'|\leq 1/p$, the last term disappears and we are left with the desired estimates \eqref{eq:AgmonGamma}.
\end{proof}

\subsubsection{Localization in impulsion}
We denote by $\xi$ the dual variable to $x$ (in the coordinates $(x,\ct)$). To establish localization for the impulsion variable $\xi$, it is easier to prove first
a localization with respect to the operator $\Xi(\fx,D_{\fx})$ is the coordinates $(\fx,\ft)$. We consider the compact set 
\begin{equation*}
    \bigcup_{j\geq 1}\{\xi\in\R\,:\,\delta_{c}^{2/3}\mu_{j}(\xi \delta_{*}^{-1/3}) \in [0,E]\} \subset [\xi_{min},\xi_{max}] =: \tilde{K}_E.
\end{equation*}
The compacity of $\tilde{K}_E$ follows from the properties of the dispersive curves $\mu_j$ resumed in Section \ref{subsect:Montgomery}. 
To quantify this, we introduce a smooth cutoff function $\tilde{\chi}$ with values in $[0,1]$ such that $\tilde{\chi}=0$ near $\tilde{K}_E$ and 1 away from $\tilde{K}_E$.

\begin{lemma}
    \label{lem:localizationfrequence}
    Consider an eigenpair $(\mu_h,\psi_h)$ of $\magSch$ with $\mu_h \leq E h^{4/3}$. Then, considering the function 
    \begin{equation*}
        \phi_\hbar = \mathfrak{U}_{\hbar} \check{\psi}_\hbar,
    \end{equation*}
    where $\mathfrak{U}_\hbar$ is the norm-preserving operator corresponding to the change of coordinates $(\fx,\ft)$, we have 
    \begin{equation}
        \label{eq:localizationfrequence}
        \tilde{\chi}(\hbar\Xi(\fx,D_{\fx})) \phi_{\hbar} = \O(\hbar^\infty).
    \end{equation}
\end{lemma}

\begin{proof}
    We will only prove the result when $\tilde{\chi}$ is 0 near $(-\infty,\xi_{max}+\frac{\eps}{2})$ and 1 near $(\xi_{max}+\eps,+\infty)$, the estimate on $(-\infty,\xi_{min}-\eps)$ following from similar arguments.
    We introduce another smooth cutoff function $\underline{\tilde{\chi}}$ with the same properties as $\tilde{\chi}$ and satisfying ${\rm supp}(\tilde{\chi})\subset {\rm supp}(\underline{\tilde{\chi}})$.
    In order to lighten the notation, we write
    \begin{equation*}
        \tilde{\chi}(\hbar\Xi) = \tilde{\chi}(\hbar\Xi(\fx,D_{\fx})).
    \end{equation*} 
    Recall that by Proposition \ref{prop:AgmonGamma}, the functions $\phi_\hbar$ are quasi-modes for the operator $\mathfrak{L}_\hbar$.
    Observe also that Lemma \ref{lem:weightedAgmon} gives an exponential decreasing of the quasi-modes along the transverse derivative $\ft$. For this reason, we introduce a cutoff function 
    $\theta\in C_{c}^\infty(\R,[0,1])$, equal to 1 in a neighborhood of 0, and $\eta\in(0,1/3)$. Now, considering the functions $\tilde{\phi}_\hbar = \theta(\hbar^\eta \ft)\phi_\hbar$,
    we still have 
    \begin{equation}
        \label{eq:quasimode2}
        \mathfrak{L}_{\hbar} \tilde{\phi}_{\hbar} = \check{\mu}_{\hbar}\tilde{\phi}_{\hbar} + \O(\hbar^\infty).
    \end{equation}
    Then, we write 
    \begin{equation}
        \label{eq:localizationfrequenceproof}
        \left(\mathfrak{L}_{\hbar} -\check{\mu}_\hbar\right)\tilde{\chi}(\hbar\Xi)\tilde{\phi}_{\hbar} = [\mathfrak{L}_{\hbar},\tilde{\chi}(\hbar\Xi)]\tilde{\phi}_{\hbar} + \O(\hbar^\infty).
    \end{equation}
    Recall the expression of the operator $\mathfrak{L}_{\hbar}$:
    \begin{equation*}
        \mathfrak{L}_{\hbar} = \delta^{2/3}\left[D_{\ft}^2+ \left(\mathfrak{m}_{\hbar}^{-1/2}(\hbar\Xi -\delta^{-1/3}\mathfrak{A}_{\hbar} + \frac{\hbar}{6}\delta'\delta^{-4/3}(i+\ft D_{\ft}+D_{\ft}\ft))\mathfrak{m}_{\hbar}^{-1/2}\right)^2\right] -\hbar^2\frac{k}{\mathfrak{m}_\hbar}.
    \end{equation*}
    Thanks to this explicit form and to the localization result in the $\fx$ variable given by Proposition \ref{prop:AgmonGamma} for $\phi_{\hbar}$, we can write for a certain constant $C>0$ that
    \begin{equation}
        \label{eq:estimbracket}
        \| [\mathfrak{L}_{\hbar},\tilde{\chi}(\hbar\Xi)]\tilde{\phi}_{\hbar}\| \leq C\hbar \|\underline{\tilde{\chi}}(\hbar\Xi)\tilde{\phi}_{\hbar}\| + C\hbar\|\underline{\tilde{\chi}}(\hbar\Xi)D_{\ft}\tilde{\phi}_{\hbar}\| + \O(\hbar^\infty)\|\tilde{\phi}_{\hbar}\|.
    \end{equation}
    Furthermore, we can also use the expression of $\mathfrak{L}_{\hbar}$ to decompose it as follows
    \begin{equation*}
        \mathfrak{L}_{\hbar} = \mathfrak{L}_{\hbar,0}+\hbar \mathfrak{R}_{\hbar},
    \end{equation*}
    where 
    \begin{equation*}
        \mathfrak{L}_{\hbar,0} = \delta(\fx)^{2/3}\left(D_{\ft}^2+ \mathfrak{P}_{\hbar,0}^2\right),
    \end{equation*}
    and where we introduced  
    \begin{equation*}
        \mathfrak{P}_{\hbar,0} = \mathfrak{m}_{\hbar}^{-1/2}\left(\hbar \Xi(\fx,D_{\fx}) -\frac{1}{2}\ft^2\right)\mathfrak{m}_{\hbar}^{-1/2},
    \end{equation*}
    and the remainder $\mathfrak{R}_{\hbar}$ can be written as 
    \begin{equation}
        \label{eq:remainder}
        \mathfrak{R}_{\hbar} =  \mathfrak{R}_{\hbar,1} + \hbar\mathfrak{R}_{\hbar,2}
    \end{equation}
    with 
    \begin{equation*}
        \mathfrak{R}_{\hbar,1} = -2\mathfrak{P}_{\hbar,0}\mathfrak{m}_{\hbar}^{-1}\left(\kappa\delta^{-1}\ft^3-\frac{1}{6}\delta'\delta^{-1}(i+\ft D_{\ft} + D_{\ft}\ft)\right),
    \end{equation*}
    and
    \begin{multline*}
        \mathfrak{R}_{\hbar,2} = \left(\mathfrak{m}_{\hbar}^{-1}\delta^{1/3}\left(\kappa\delta^{-1}\ft^3-\frac{1}{6}\delta'\delta^{-1}(i+\ft D_{\ft} + D_{\ft}\ft)\right) + \hbar\delta^{2/3}\mathfrak{r}_{\hbar}\right)^2 
        - 2 \mathfrak{P}_{\hbar,0} \mathfrak{m}_{\hbar}^{-1}\mathfrak{r}_{\hbar} - \frac{k}{\mathfrak{m}_\hbar}.
    \end{multline*}
    We introduce an increasing smooth function $\xi\in\R\mapsto \chi^\sharp(\xi)$ that coincides with $\nu_{max}+\frac{\eps}{4}$ on $(-\infty,\nu_{max}+\frac{\eps}{4})$ and with the identity
    on $(\nu_{max}+\frac{\eps}{2},+\infty)$. We define the operator
    \begin{equation*}
        \mathfrak{L}_{\hbar,0}^{cut} = \delta^{2/3}\left(D_{\ft}^2+ \left(\mathfrak{m}_{\hbar}^{-1/2}\left(\chi^\sharp(\hbar \Xi) -\frac{1}{2}\ft^2\right)\mathfrak{m}_{\hbar}^{-1/2}\right)^2\right),
    \end{equation*}
    acting on the domain
    \begin{equation*}
        \dom(\mathfrak{L}_{\hbar,0}^{cut}) = \left\{u \in L^2(\R^2)\,:\, \mathfrak{L}_{\hbar,0}^{cut}u \in L^2(\R^2) \right\}.
    \end{equation*}

    We notice by a simple application of the min-max theorem that $\mathfrak{L}_{\hbar,0}^{cut}-\check{\mu}_\hbar$ is invertible due to the choice of the cutoff function $\chi^\sharp$ and the definition of $\xi_{max}$.
    Moreover, since ${\rm supp}(\tilde{\chi}) \subset \{\chi^\sharp  = {\rm Id}\}$, using Equation \eqref{eq:localizationfrequenceproof} we have 
    \begin{equation*}
        \left(\mathcal{L}_{\hbar,0}^{cut}-\check{\mu}_\hbar + \mathfrak{R}_{\hbar}\right)\tilde{\chi}(\hbar\Xi)\tilde{\phi}_\hbar = [\mathcal{L}_{\hbar},\tilde{\chi}(\hbar\Xi)]\tilde{\phi}_{\hbar} + \O(\hbar^\infty),
    \end{equation*}
    which can be written as 
    \begin{equation*}
        \left({\rm Id}+ \mathfrak{R}_{\hbar}(\mathfrak{L}_{\hbar,0}^{cut}-\check{\mu}_\hbar)^{-1}\right)(\mathfrak{L}_{\hbar,0}^{cut}-\check{\mu}_\hbar)\tilde{\chi}(\hbar\Xi)\tilde{\phi}_\hbar = [\mathfrak{L}_{\hbar},\tilde{\chi}(\hbar\Xi)]\tilde{\phi}_{\hbar} + \O(\hbar^\infty).
    \end{equation*}
    By using Equation \eqref{eq:remainder} and  definition of the domain of $\mathfrak{L}_{\hbar,0}^{cut}$, we get that 
    \begin{equation*}
        \|\mathfrak{R}_{\hbar}(\check{\mathcal{L}}_{\hbar,0}^{cut}-\check{\mu}_\hbar)^{-1}\|_{L^2\rightarrow L^2} = \O(\hbar^{1-3\eta}),
    \end{equation*}
    as the polynomial term in $\ft$ have a bounded action on $\tilde{\phi}_\hbar$ by the cutofff function $\theta(\hbar^\eta\cdot)$.

    The operator ${\rm Id}+\mathfrak{R}_{\hbar}(\mathfrak{L}_{\hbar,0}^{cut}-\check{\mu}_\hbar)^{-1}$ is then bijective as soon as $\hbar$ is small enough.
    With Equation \eqref{eq:estimbracket}, we obtain 
    \begin{equation*}
        \|\tilde{\chi}(\hbar\Xi)\tilde{\phi}_{\hbar}\|^2\leq C\hbar \|\underline{\tilde{\chi}}(\hbar\Xi)\tilde{\phi}_{\hbar}\| + C\hbar \|\underline{\tilde{\chi}}(\hbar\Xi)D_{\ft}\tilde{\phi}_{\hbar}\|+\O(\hbar^\infty)\|\tilde{\phi}_{\hbar}\|^2.
    \end{equation*}
    Using Equation \eqref{eq:quasimode2} once again, we find 
    \begin{equation*}
        \|\tilde{\chi}(\hbar\Xi)\tilde{\phi}_{\hbar}\|^2+\|\tilde{\chi}(\hbar\Xi)D_{\ft}\tilde{\phi}_{\hbar}\|\leq C\hbar \|\underline{\tilde{\chi}}(\hbar\Xi)\tilde{\phi}_{\hbar}\| + C\hbar \|\underline{\tilde{\chi}}(\hbar\Xi)D_{\ft}\tilde{\phi}_{\hbar}\|+\O(\hbar^\infty)\|\tilde{\phi}_{\hbar}\|^2.
    \end{equation*}
    The estimate \eqref{eq:localizationfrequence} follows by an induction argument on the support of the cutoff functions $\tilde{\chi}$ and $\underline{\tilde{\chi}}$, and from the fact 
    that $\phi_\hbar$ and $\tilde{\phi}_\hbar$ coincide up to $\O(\hbar^\infty)$ for the graph norm of $\mathfrak{L}_\hbar$.
\end{proof}

From the previous lemma, we obtain the desired localization with respect to impulsion $D_x$.

\begin{proposition}
    \label{prop:localizationfrequence}
    Let $E\in(-\infty,E_0)$. There exist $h_0 >0$ and $\check{\chi} \in C^{\infty}(\R)$ satisfying $0\leq \check{\chi}\leq 1$ and equal to 0 on a neighborhood of 0, such that, for all $h\in(0,h_0)$ and for all eigenpair $(\mu_h,\psi_h)$ of $\magSch$ with $\mu_h\leq E h^{4/3}$, we have 
    \begin{equation}
        \label{eq:proplocalizationfrequence}
        \check{\chi}(\hbar D_{x})\check{\psi}_{\hbar} = \O(\hbar^{\infty}).
    \end{equation}
    Moreover, we have 
    \begin{equation*}
        \check{\cQ}_{\hbar}\left(\check{\chi}(\hbar D_{x})\check{\psi}_{\hbar}\right) = \O(\hbar^\infty).
    \end{equation*}
\end{proposition}

\begin{proof}
    First, we can give the following explicit expression of the pull-back of the operator $\hbar\Xi(\fx,D_{\fx})$ in the coordinates $(x,\ct)$:
    \begin{equation*}
        \delta(x)^{-1/3}\hbar D_{x} + \frac{1}{3}\delta'(x)\delta(x)^{-4/3} \ct \hbar D_{\ct}.
    \end{equation*}
    This last operator is a self-adjoint semiclassical pseudodifferential operator in both the variables $x$ and $\ct$, whose symbol we will simply denote by $a$:
    \begin{equation*}
        \Op_{\hbar}^{w}(a) = \delta(x)^{-1/3}\hbar D_{x} + \frac{1}{3}\delta'(x)\delta(x)^{-4/3} \ct \hbar D_{\ct}.
    \end{equation*}
    By Lemma \ref{lem:localizationfrequence}, we have 
    \begin{equation}
        \label{eq:localizationfrequenceproof}
        \tilde{\chi}(\Op_{\hbar}^{w}(a))\check{\psi}_{\hbar} = \O(\hbar^{\infty}).
    \end{equation}
    We observe that to get a localization in impulsion $D_x$, we need to understant the localization in impulsion $D_{\ct}$, at the semiclassical scale $\hbar$, of our eigenfunction.
    We introduce a new smooth cutoff function $\chi^{\sharp}\in C_{0}^{\infty}(\R)$ with $0\leq \chi^{\sharp}\leq 1$ equal to 1 near 0. 
    For $\eta \in (0,1)$, by consideration on the support of the cutoff  we have that
    \begin{equation*}
        \left\|\left(1-\chi^{\sharp}(\hbar^{1-\eta} D_{\ct})\right)\check{\psi}_{\hbar}\right\| \leq \left\|\hbar^{1-\eta} D_{\ct}\left(1-\chi^{\sharp}(\hbar^{1-\eta} D_{\ct})\right)\check{\psi}_{\hbar}\right\| \leq \hbar^{1-\eta} \|D_{\ct}\check{\psi}_{\hbar}\|.
    \end{equation*}
    And since $\check{\psi}_\hbar$ is a quasi-mode for $\check{\mathcal{L}}_\hbar$, the last term is bounded by $E\|\check{\psi}_\hbar\|$:
    \begin{equation*}
        \left\|\left(1-\chi^{\sharp}(\hbar^{1-\eta} D_{\ct})\right)\check{\psi}_{\hbar}\right\| \leq E\hbar^{1-\eta} \|\check{\psi}_{\hbar}\|.
    \end{equation*}
    Moreover, for any $\eta' \in (0,\eta)$, as a corollary of Lemma \ref{lem:weightedAgmon} we can see that 
    \begin{equation*}
        \chi^{\sharp}(\hbar^{\eta'}\ct)\check{\psi}_{\hbar} = \check{\psi}_{\hbar} + \O(\hbar^\infty).
    \end{equation*}
    We are now ready to prove Equation \eqref{eq:proplocalizationfrequence} and determine the support of $\check{\chi}$. By Equation \eqref{eq:localizationfrequenceproof}, we have
    \begin{equation*}
        \begin{aligned}
        \check{\chi}(\hbar D_{x}) \check{\psi}_{\hbar}  &= \check{\chi}(\hbar D_{x})\left(1-\tilde{\chi}(\Op_{\hbar}^{w}(a))\right)\check{\psi}_{\hbar} + \O(\hbar^{\infty})\\
        &=\check{\chi}(\hbar D_{x})\left(1-\tilde{\chi}(\Op_{\hbar}^{w}(a))\right)\chi^{\sharp}(\hbar^{\eta'}\ct)\left(1-\chi^{\sharp}(\hbar^{1-\eta}D_{\ct})\right)\check{\psi}_{\hbar} \\
        &+\check{\chi}(\hbar D_{x})\left(1-\tilde{\chi}(\Op_{\hbar}^{w}(a))\right)\chi^{\sharp}(\hbar^{\eta'}\ct)\chi^{\sharp}(\hbar^{1-\eta}D_{\ct})\check{\psi}_{\hbar}  +\O(\hbar^{\infty}).
        \end{aligned}
    \end{equation*}
    We observe that the first term in the right-hand side of the last equation can be bounded by
    \begin{equation*}
        \hbar^{1-\eta}\|\underline{\check{\chi}}(\hbar D_x)\check{\psi}_{\hbar}\|,
    \end{equation*}
    where $\underline{\check{\chi}}$ is a cutoff function with similar properties as $\check{\chi}$ but equal to 1 on ${\rm supp}(\check{\chi})$. For the main term, we look at the intersection of the supports of the symbols of the 
    pseudodifferentials operators $1-\tilde{\chi}(\Op_{\hbar}^{w}(a))$, $\chi^\sharp(\hbar^{\eta'}\ct)$ and $\chi^\sharp(\hbar^{1-\eta}D_{\ct})$: noting $\zeta$ the dual variable to $\ct$, and $\xi$ the dual variable of $x$, we must necessarily have 
    \begin{equation*}
        \delta(x)^{-1/3}\xi + \frac{1}{3}\delta'(x)\delta(x)^{-4/3} \ct \zeta \in [\xi_{min}, \xi_{max}],\ |\ct| \leq \hbar^{-\eta'},\ |\zeta| \leq \hbar^{\eta}.
    \end{equation*}
    As $\eta > \eta'$, we obtain by localization in the varibale $x$, that the variable $\xi$ lives inside a compact set whose definition depends of $\xi_{min}$, $\xi_{max}$ and $\chi^{\sharp}$. Choosing $\check{\chi}$ with support away from this compact set,
    we obtain by pseudodifferential calculus that 
    \begin{equation*}
        \|\check{\chi}(\hbar D_{x}) \check{\psi}_{\hbar}\| \leq \hbar^{1-\eta}\|\underline{\check{\chi}}(\hbar D_x)\check{\psi}_{\hbar}\| + \O(\hbar^{\infty}).
    \end{equation*}
    By induction on the supports of $\check{\chi}$ and $\underline{\check{\chi}}$, we obtain the desired estimate.
\end{proof}

\subsection{Reduction to a bounded pseudodifferential operator}
\label{subsect:reducbounded}

Let $E\in(-\infty,E_0)$. Thanks to the previous localization estimates, we are able to exhibit a pseudodifferential operator $\hbar^{4}\mathcal{N}_\hbar$ whose spectrum is close to that of $\magSch$ up to $\O(\hbar^\infty)$ in the energy window $(-\infty,Eh^{4/3}]$.

\subsubsection{A pseudodifferential operator with operator-valued symbol}

Our starting point is the expression of $\check{\mathcal{L}}_\hbar$, the magnetic Laplacian in dilated tubular coordinates $(x,\ct)$ given by Equation \eqref{eq:rescaledmagSch},
which we recall here:
\begin{equation*}
    \check{\mathcal{L}}_\hbar = D_{\ct}^2 + \left(m_{\hbar}^{-1/2}\left(\hbar D_{x} - \check{A}_{\hbar}(x,\ct)\right)m_{\hbar}^{-1/2}\right)^2
    -\hbar^{2}\frac{k(x)^2}{4m_{\hbar}^2}.
\end{equation*}
This expression makes sense for test function supported in $\R\times (-\hbar^{-1}d_0,\hbar^{-1}d_0)$, and as a first step, we wish to define an operator on the full plane.
To this end, we introduce a cutoff function $\theta \in C_{c}^\infty(\R)$ equal to 1 in a neighborhood of 0 and real $\eta$ in $(0,1)$.
We also consider a smooth compactly supported function $\mathring{\chi}$ that is equal to 1 on a neighborhood of the compact set $K_{E}$ introduced in Proposition \ref{prop:AgmonGamma}, for some small $\eps>0$, 
and we introduce a smooth bounded function $\mathring{\delta}$ such that 
\begin{equation*}
    \forall x\in K_E,\ \mathring{\delta}(x) = \delta(x)\quad\mbox{and}\quad \forall x\in\R\setminus K_E,\ \mathring{\delta}(x) > \sup_{K_E} \delta,
\end{equation*}
which exists by definition of the compact set $K_E$, as well as a new magnetic potential
\begin{equation*}
    \mathring{A}_{\hbar}(x,\ct) = \frac{1}{2}\mathring{\delta}(x)\ct^2 + \hbar  \mathring{\kappa}(x)\theta(\hbar^\eta \ct)\ct^3 + \hbar^2 \mathring{r}_\hbar(x,\ct),
\end{equation*} 
where 
\begin{equation*}
    \mathring{\kappa}(x) = \mathring{\chi}(x)\kappa(x)\quad\mbox{and}\quad \mathring{r}_\hbar(x,\ct) = \mathring{\chi}(x)r_\hbar(x,\theta(\hbar^\eta \ct)\ct).
\end{equation*}
We also make a choice of smooth function $\mathring{\Xi}(\xi)$ that is increasing and bounded, as well as coincide with the identity on a neighborhood of the support of $1-\check{\chi}$, where $\check{\chi}$ 
is the smooth function introduced in Proposition \ref{prop:localizationfrequence}. We can now define the desired operator:
\begin{equation*}
    \mathcal{N}_\hbar = D_{\ct}^2 + \left(\mathring{m}_{\hbar}^{-1/2}\left(\mathring{\Xi}(\hbar D_{x}) - \mathring{A}_{\hbar}(x,\ct)\right)\mathring{m}_{\hbar}^{-1/2}\right)^2
    -\hbar^{2}\frac{k(x)^2}{4\mathring{m}_{\hbar}^2},
\end{equation*}
where we have set 
\begin{equation*}
    \mathring{m}_{\hbar}(x,\ct) = 1 - \hbar k(x) \theta(\hbar^\eta \ct)\ct.
\end{equation*}
For $\hbar$ small enough, these cutoffs allow the operator $\mathcal{N}_\hbar$ to be defined on the whole plane $\R^2$.

This construction has been guided by the fact that it is possible to see the operator $\mathcal{N}_\hbar$ as a semiclassical pseudodifferential operator with operator-valued symbol, in the sense of Appendix \ref{sect:pseudoropvalued}. To see this, we recall the definition of the Hilbert space $B^{2,4}(\R)$:
\begin{equation*}
    B^{2,4}(\R) = \{u\in L^2(\R)\,:\, u''\in L^2(\R),\ \langle \ct\rangle^4 u \in L^2(\R)\},
\end{equation*}
and endowed with the norm 
\begin{equation*}
    \|u\|_{B^{2,4}(\R)} = \|u''\|_{L^2(\R)} + \|\langle \ct\rangle^4 u\|_{L^2(\R)}.
\end{equation*}

\begin{lemma}
    \label{lem:Nhbarpseudo}
    The operator $\mathcal{N}_\hbar$ can be written as the Weyl quantization $n_{\hbar}^w$ of a symbol $n_\hbar$ in $S(\R^2,\mathscr{L}(B^{2,4}(\R),L^2(\R)))$ admitting the semiclassical expansion
    \begin{equation}
        \label{eq:nhbarexpansion}
        n_\hbar = \sum_{j=0}^\infty \hbar^j n_j,\quad n_j \in S^{-j\eta}(\R^2,\mathscr{L}(B^{2,4}(\R),L^2(\R)))
    \end{equation}
    where the principal symbol $n_0$ is given by
    \begin{equation*}
        n_0(x,\xi) = D_{\ct}^2 + p_0^2(x,\xi),
    \end{equation*}
    and the subprincipal symbol $n_1$ by
    \begin{equation*}
        n_1(x,\xi) = -2 \mathring{\kappa}(x)\theta(\hbar^\eta\ct)\ct^3 p_0(x,\xi) + 2k(x) \theta(\hbar^\eta\ct)\ct\, p_0^2(x,\xi),
    \end{equation*}
    where
    \begin{equation*}
        p_0(x,\xi) = \mathring{\Xi}(\xi) - \frac{1}{2}\mathring{\delta}(x)\ct^2.
    \end{equation*}
\end{lemma}

\begin{proof}
    Observe that we can write 
    \begin{equation*}
        \mathring{\Xi}(\hbar D_x) - \mathring{A}_{\hbar}(x,\ct) = \Op_{\hbar}^w(p_\hbar),
    \end{equation*}
    with 
    \begin{equation*}
        p_\hbar(x,\xi) = p_0(x,\xi) - \hbar \mathring{\chi}(x)\kappa(x)\theta(\hbar^\eta \ct)\ct^3 - \hbar^2 \mathring{\chi}(x)r_\hbar(x,\theta(\hbar^\eta \ct)\ct).
    \end{equation*}
    We easily check that $p_0$ is a symbol in $S^0(\R^2,\mathscr{L}(B^{2,4}(\R),L^2(\R)))$, while the symbol 
    \begin{equation*}
        \mathring{\kappa}(x)\theta(\hbar^\eta \ct)\ct^3,
    \end{equation*}
    is in the symbol class $S^{-\eta}(\R^2,\mathscr{L}(B^{2,4}(\R),L^2(\R)))$ since the multiplication by $\theta(\hbar^\eta \ct)\ct^3$ is a bounded operator from $B^{2,4}(\R)$ to $L^2(\R)$ due to the cutoff $\theta$, with norm $\hbar^{-\eta}$. Similarly, the symbol $\mathring{r}_\hbar(x,\ct)$ can be developed as a semiclassical expansion with symbols in classes $S^{-j\eta}$.

    Using the same arguments for the operator $\mathring{m}_{\hbar}^{-1/2}$ using the Taylor expansion of the squared root function, the operator $\mathcal{N}_\hbar$ can be written as the composition of semiclassical pseudodifferential operators with operator-valued symbols that admit semiclassical expansions with symbols in the classes $S^{-j\eta}$. By composition, in particular by Proposition \ref{prop:symb_cal} in Appendix \ref{sect:pseudoropvalued}, $\mathcal{N}_\hbar$ is itself a semiclassical pseudodifferential operator with a symbol as in Equation \eqref{eq:nhbarexpansion}.

    The principal symbols $n_0$ and $n_1$ are obtained through direct computation and application of Proposition \ref{prop:symb_cal}, once we notice that we have $\{p_0,p_0\}(x,\xi) = 0$.
\end{proof}

The next lemma allows us to talk about the spectrum of the operator $\mathcal{N}_\hbar$.

\begin{lemma}
    \label{lem:Nhbarautoadjoint}
    The operator $\mathcal{N}_\hbar$ is essentially self-adjoint on $\mathcal{S}(\R,B^{2,4}(\R))\subset L^2(\R^2)$ with domain given by $L^2(\R,B^{2,4}(\R))$.
\end{lemma}

This follows from the fact that $\mathcal{N}_\hbar$ is formally symmetric and from the ellipticity of the symbol $n_0\pm i$ in $S^0(\R^2,\mathscr{L}(B^{2,4},L^2(\R)))$, in the sense of Definition \ref{def:symelliptic}.

The proof of Theorem \ref{thm:spectrediscret} then applies to the operator $\mathcal{N}_\hbar$ with minor modifications. This leads to the following proposition
\begin{proposition}
    \label{prop:spectrediscretN}
    Let $E\in(-\infty,E_0)$. There exists $h_0>0$ such that, for $\hbar \in(0,\hbar_0)$ the spectrum of $\mathcal{N}_\hbar$ in the energy window is purely discrete.
\end{proposition}

\subsubsection{Equivalence of spectra}

To prove that the two spectra coincide, we need to prove that we can construct quasi-modes of $\magSch$ from eigenstates of $\mathcal{N}_\hbar$: this is possible thanks to the following lemma,
which states the exponential decreasing with respect to $\ct$ of such eigenstates.

\begin{lemma}
    \label{lem:weightedAgmonbis}
    Let $E \in(-\infty,E_0)$ and $M>0$. There exist $h_0>0$ and $C>0$ such that, for all eigenpairs $(\mathring{\mu}_h,\mathring{\psi}_\hbar)$ of $\mathcal{N}_\hbar$ satisfying $\mathring{\mu}_\hbar\leq E $, we have for all $h\in (0,h_0)$:
    \begin{equation*}
        \begin{aligned}
            \|e^{M |\ct|}\mathring{\psi}_{\hbar}\| &\leq C\|\mathring{\psi}_{\hbar}\|,\\
            \mathring{\cQ}_{\hbar}\left(e^{M |\ct|}\mathring{\psi}_{\hbar}\right)&\leq C\|\mathring{\psi}_{\hbar}\|^2,
        \end{aligned}
    \end{equation*}
    where $\mathring{\cQ}_{\hbar}$ denotes the sesquilinear form associated to $\mathcal{N}_\hbar$. 
\end{lemma}

We do not write the proof of this lemma as it is coincides with the proof of Lemma \ref{lem:weightedAgmon}, up to minor changes.

\begin{proposition}
    \label{prop:equivspectra}
    Let $E\in(-\infty,E_0)$. There exists $h_0 >0$ such that for all $h\in (0,h_0)$, the spectra of $\magSch$ and $h^{4/3}\mathcal{N}_\hbar$ coincide in 
    the interval $I_h = (-\infty,Eh^{4/3})$ up to $\mathscr{O}(h^\infty)$.
\end{proposition}

\begin{proof}
    Let us start by proving that for any $\lambda\in I_h \cap \sigma(\magSch)$, we have 
    \begin{equation}
        \label{eq:quasimodeN}
        {\rm dist}(\lambda,h^{4/3}\sigma(\mathcal{N}_\hbar)) = \O(h^\infty).
    \end{equation}
    For $\psi$ an eigenfunction corresponding to $\lambda$, by definition of the interval $I_h$ and Propositions \ref{prop:AgmonGamma} and \ref{prop:localizationfrequence},
    as well as the construction of the operator $\mathcal{N}_\hbar$, we find
    \begin{equation*}
        h^{4/3}\mathcal{N}_\hbar \check{\psi} = \lambda \check{\psi} + \O(h^\infty),
    \end{equation*}
    where $\check{\psi}$ is given by Equation \eqref{eq:checkpsi}. Thus, Equation \eqref{eq:quasimodeN} follows from the spectral theorem.

    Reciprocally, if $\lambda$ is in $I_h \cap h^{4/3}\sigma(\mathcal{N}_\hbar)$, and denoting by $\mathring{\psi}\in L^2(\R_x\times\R_{\ct})$ a corresponding eigenfunction,
    we can consider the function $\psi$ defined in tubular coordinates $(x,t)$, introduced by Equation \eqref{eq:standardtubucoord}, by 
    \begin{equation*}
        \psi(x,t) = \hbar^{-1/2} \theta(\hbar^{\eta-1} t)\psi(x,\hbar^{-1}t).
    \end{equation*}
    By Lemma \ref{lem:weightedAgmonbis}, we deduce that $\psi$ is a quasi-mode for $\magSch$:
    \begin{equation*}
        \magSch \psi = \lambda \psi + \O(h^\infty),
    \end{equation*}
    and the spectral theorem gives once again 
    \begin{equation*}
        {\rm dist}(\lambda,\sigma(\magSch)) = \O(h^\infty).
    \end{equation*}
    Thus, to prove that the two spectra coincide, we need only to deal with multiplicities.
    Following Definition \ref{def:spectracoincide}, we consider a $h$-dependent interval $J_h \subset I_h$ and we write 
    $J_h\cap\sigma(\magSch) = \{\lambda_1, \dots ,\lambda_p\}$ (where the $\lambda_j$ are distinct). We emphasize the fact that 
    these eigenvalues depend on the semiclassical parameter $h$ as does $p$. We also consider the associated eigenspaces $(E_j)_{1\leq j\leq p}$
    and we note that $\dim \bigoplus_{j=1}^p E_j = \O(h^{-2})$ thanks to the rough Weyl estimate given by Corollary \ref{cor:roughweyllaw1}.
    
    Denoting by $(\check{E}_j)_{1\leq j\leq p}$ the spaces of quasimodes of $h^{4/3}\mathcal{N}_\hbar$ obtained by considering the images of the eigenspaces 
    $E_j$ under the transformation given by Equation \eqref{eq:checkpsi}, we observe that by Proposition \ref{prop:localizationGamma}, we have 
    \begin{equation*}
        \dim \check{E}_j = \dim E_j,
    \end{equation*}
    as soon as $h$ is small enough. Moreover, as the dimension of $\bigoplus_{j=1}^p \check{E}_j$ is controlled by $\O(h^{-2})$, we have 
    \begin{equation}
        \label{eq:quasimodefull}
        \|(\oplus_{j=1}^p h^{4/3}\mathcal{N}_\hbar - \lambda)\psi\| \leq \eps_h \|\psi\|,\quad \eps_h = \O(h^\infty),
    \end{equation}
    for all $\psi = (\psi_1,\dots,\psi_p)\in \bigoplus_{j=1}^p \check{E}_j$, and where $\lambda = (\lambda_1,\dots,\lambda_p)$.
    Setting $J_h = [a_h,b_h]$, we consider $\check{J}_h = [a_h-\eps_h,b_h+\eps_h]$. We wish to prove that we have 
    \begin{equation}
        \label{eq:spectruminclusion}
        {\rm rank} \mathbb{1}_{J_h}(\magSch) \leq {\rm rank} \mathbb{1}_{\check{J}_h}(h^{4/3}\mathcal{N}_\hbar).
    \end{equation}
    If we had ${\rm rank} \mathbb{1}_{\check{J}_h}(h^{4/3}\mathcal{N}_\hbar) < {\rm rank} \mathbb{1}_{J_h}(\magSch)$, then the projection 
    $\Pi: \bigoplus_{j=1}^p \check{E}_j \rightarrow {\rm ran} \mathbb{1}_{\check{J}_h}(h^{4/3}\mathcal{N}_\hbar)$ could not be injective.
    Considering a non-zero $\psi$ in its kernel, the spectral theorem would give $\|(\oplus_{j=1}^p h^{4/3}\mathcal{N}_\hbar - \lambda)\psi\| > \eps_h \|\psi\|$,
    which is in contradiction with Equation \eqref{eq:quasimodefull}, as $\psi\neq 0$. Therefore, inequality \eqref{eq:spectruminclusion} holds.
    
    We prove the reciprocal inclusion of the spectrum of $h^{4/3}\mathcal{N}_\hbar$ (with multiplicities) into the one of $\magSch$ by similar considerations, using also in this case a rough Weyl estimate given by Corollary \ref{cor:roughweyllaw2}.
\end{proof}

%%%%%%%%%%%%%%%%%%%%%%%%%%%%%%%%%%%%%%%%%%%%%%%%%%%%%%%%%%%%%%%%%%%%%%%%%%%%%%%%%%%%%%%%%%%%%%%%%%%%%%%%%%%%%%%%%%%%%%%%%%%%%%

\section{Superadiabatic projectors and effective Hamiltonians}
\label{sect:dimreduction}

The preceding sections have shown (modulo the proof of Corollary \ref{cor:roughweyllaw2} presented in this section) the equivalence of the spectral analysis of the magnetic Laplacian $\magSch$ to that of the pseudodifferential operator 
\begin{equation*}
    \hbar^4 \mathcal{N}_\hbar = \hbar^4 n_{\hbar}^w,
\end{equation*}
with principal symbol $n_0$ given by a reparametrization of the Montgomery operators in the variable $(x,\xi)\in\R^{2}$. 
We present now a general scheme, inspired by the Born-Oppenheimer approximation and developed in \cite{BFKRV}, which applied to the operator $\mathcal{N}_\hbar$, reduce its spectral study to that of a finite number of one-dimensional 
scalar semiclassical pseudodifferential operators, each one related to a dispersion curve $\mu_k(\cdot)$, $k\in\N$, of the Montgomery operators $(\mathfrak{M}(\nu))_{\nu\in\R}$.

\subsection{A general reduction scheme}
\label{subsect:reductionscheme}

In this section we fix two Hilbert spaces $\A$ and $\B$, with the assumption that $\A$ is continuously embedded in $\B$. We suppose given an operator-valued self-adjoint symbol $H_{\hbar}\in S^0(\R^{2n},\mathscr{L}(\A,\B))$, as in Appendix \ref{sect:pseudoropvalued}, admitting
the following semiclassical asymptotic expansion
\begin{equation*}
    H_{\hbar} \sim \sum_{j=0}^{\infty}\hbar^{j} H_{j}.
\end{equation*}
As we are interested in the spectral analysis of the operator $H_{\hbar}^w$, it is natural to look for (almost) invariant subspaces. 
To do so, the idea consists in using the pseudodifferential calculus to construct orthogonal projectors that (almost) commute with $H_{\hbar}^w$. 
First, recall that, given two operator-valued symbols $p$ and $q$ whose composition makes sense, the composition of the two Weyl quantizations $p^w$ and $q^w$ 
is still a pseudodifferential operator, whose symbol is given by the Moyal product denoted by $p\circledast q$:
\begin{equation*}
    (p \circledast q)^w = p^w \circ q^w.
\end{equation*}
Thus, we are looking for admissible symbols $\Pi_\hbar$ in $S^0(\R^{2n},\mathscr{L}(\A))\cap S^0(\R^{2n},\mathscr{L}(\B))$ with semiclassical asymptotic expansion
\begin{equation}
    \Pi_\hbar \sim \sum_{j=0}^{\infty} \hbar^{j}\Pi_{j}(x,\xi),
\end{equation}
satisfying
\begin{equation}
    \label{eq:symbolicinvariantsubspace}
        \Pi_\hbar = \Pi_{\hbar}^{*},\quad \Pi_\hbar \sim \Pi_\hbar \circledast \Pi_\hbar,
        \quad\mbox{and}\quad H_{\hbar}\circledast \Pi_\hbar -  \Pi_\hbar \circledast H_{\hbar} \sim 0.
\end{equation}
The corresponding quantized operator $\Pi_{\hbar}^w$ is called a superadiabatic projector: it acts continuously on $L^2(\R^n,\A)$ and $L^2(\R^n,\B)$, and commutes up to $\mathscr{O}(\hbar^\infty)$ with $H_\hbar^w$.

\vspace{0.25cm}

It is possible to construct such superadiabatic projectors for each isolated part of the spectrum of the principal symbol $H_{0}$ of $H_{\hbar}$. More precisely, 
we now assume that $H_0(x,\xi)$, $(x,\xi)\in\R^{2n}$, is a self-adjoint operator on the Hilbert space $\B$ with domain $\A\subset \B$, and we denote its spectrum by $\sigma(x,\xi)$. 

Suppose that there exists a smoothly parametrized bounded subset $\sigma_0(x,\xi)\subset \sigma(x,\xi)$, $(x,\xi)\in\R^{2n}$, of the spectrum of $H_0(x,\xi)$, in the sense that the associated spectral projector $\Pi_0(x,\xi)$, $(x,\xi)\in\R^{2n}$ defines a symbols, and such that for some constant $\eps >0$ we have
\begin{equation}
    \label{eq:spectrumisolated}
    {\rm dist}_H(\sigma_{0}(x,\xi),\sigma(x,\xi)\setminus \sigma_0(x,\xi)) > \eps.
\end{equation}
Then, we have the following theorem.

\begin{theorem}[\cite{BFKRV}]
    \label{thm:superadiabproj}
    With the above assumptions, there exists an admissible operator-valued self-adjoint symbol $\Pi_\hbar\in S(\R^{2d},\mathscr{L}(\mathscr{B},\mathscr{A}))$ with principal symbol $\Pi_0$ and unique up to a remainder of order $\mathscr{O}(h^\infty)$, satisfying
    \[
    \Pi_\hbar^w\circ \Pi_\hbar^w=\Pi_\hbar^w+\mathscr{O}(\hbar^\infty),
    \]
    where the remainder $\mathscr{O}(\hbar^\infty)$ is a pseudodifferential operator whose symbol is in the same class as $\Pi$, and 
    \[
    [H_\hbar^w,\Pi_\hbar^w]=\mathscr{O}(\hbar^\infty)\,,
    \]
    where the remainder $\mathscr{O}(\hbar^\infty)$ is a pseudodifferential operator whose symbol is in the same class as $H_\hbar$.
\end{theorem}

In order to better understand the role that superadiabatic projectors play in the spectral analysis of the operator $H_\hbar^w$, we recall the following property.

\begin{proposition}[\cite{BFKRV}]
\label{prop:commutespectral}
With the same assumptions as in Theorem~\ref{thm:superadiabproj}, we have the following properties:
\begin{enumerate}
    \item[(i)]  For any function $\chi\in C_{c}^\infty(\R)$ we have 
    $[\chi(H_\hbar^w),\Pi_\hbar^w] = \mathscr{O}(\hbar^\infty)$, in $L^2(\R^d,\mathscr{B})$.
    \item[(ii)] Assume that there exists an energy level $E \in\R$ such that 
    \begin{equation}\label{eq:localspectralH0}
		\exists \eps>0,\ \sigma(\restriction{H_{0}}{{\rm Ran}(\Pi_0^\perp)}) \subset (E+\eps,+\infty).
    \end{equation}
    Then for any function $\chi\in C_{c}^\infty(\R)$ such that ${\rm supp}(\chi) \subset (-\infty,E]$, we have
	\begin{equation*}
		 \Pi_\hbar^w\circ\chi(H_\hbar^w) = \chi(H_\hbar^w) + \mathscr{O}(\hbar^\infty).
	\end{equation*}
\end{enumerate}
\end{proposition}

It follows that, under the assumption of the second point of the previous proposition, the spectral analysis of the self-adjoint operator $H_\hbar^w$ in the energy window $(-\infty,E]$ can be reduced to that of the self-adjoint operator $\Pi_\hbar^w\circ H_\hbar^w\circ \Pi_\hbar^w$. 

\subsection{Effective one-dimensional Hamiltonians}
We apply the reduction scheme exposed in the preceding section to the case of the operator $\mathcal{N}_\hbar$. 

\subsubsection{Superadiabatic projectors and partial isometries}

We recall from Section \ref{sect:intro} that the principal symbol $n_0$ is indeed, for all $(x,\xi)\in\R^2$, a self-adjoint operator on $L^2(\R)$
with domain $B^{2,4}(\R)$ and its spectrum is given by 
\begin{equation*}
    \sigma(x,\xi) = \{\mathring{\mu}_k(x,\xi):=\mathring{\delta}^{2/3}(x)\mu_{k}(\mathring{\Xi}(\xi)\mathring{\delta}^{-1/3}(x))\,:\, k\in \N\}.
\end{equation*}
It follows from the properties of the dispersive curves $(\mu_k(\cdot))_{k\in\N}$ as well as from the definitions of the functions $\mathring{\delta}$ and $\mathring{\Xi}$
that each eigenvalue $\mathring{\mu}_k$, $k\in\N$, is uniformly isolated from the rest of the spectrum in the sense of Equation \eqref{eq:spectrumisolated}. 
By Theorem \ref{thm:superadiabproj}, we deduce there exists a corresponding superadiabatic projectors that we denote by $\Pi_\hbar^k$ satisfying
\begin{equation*}
    \Pi_\hbar^k = \Pi_0^k + \O(\hbar).
\end{equation*}
Since the operators $\Pi_\hbar^k$, $k\in\N$, correspond
to two by two disjoint parts of the spectrum $\sigma(x,\xi)$, there are two by two almost orthogonals up to $\O(\hbar^\infty)$ (\cite[Proposition 1.3]{BFKRV}).

\begin{remark}
    The precise assumption of Theorem \ref{thm:superadiabproj} requires the symbol $n_\hbar$ to be admissible, which he is not since, by Lemma \ref{lem:Nhbarpseudo}, the symbols appearing in its asymptotic expansion are in the classes $S^{-j\eta}(\R^2,\mathscr{L}(B^{2,4}(\R),L^2(\R)))$, $j\in\N$. However, since this is due to the presence of the cutoffs function $\theta$, the recursive construction of the symbol $\Pi_\hbar$ proposed in \cite{BFKRV} is still valid. The symbol $\Pi_\hbar$ then admits an asymptotic expansion of the same kind as $n_\hbar$.
\end{remark}

In order to reduce our analysis to one-dimensional scalar semiclassical pseudodifferential operators, we follow the strategy of \cite{BFKRV} and introduce partial isometries associated with each superadiabatic projector.
More precisely, we recall that by Equation \eqref{eq:smootheigenfunction}, we can single out a smoothly
parametrized normalized basis $\mathring{u}_k$ for each $k\in\N$ given by 
\begin{equation*}
    \forall (x,\xi)\in\R^2,\quad \mathring{u}_k(x,\xi) = \mathcal{T}^*(x)\tilde{u}_{k}(\mathring{\Xi}(\xi)\mathring{\delta}(x)^{-1/3}).
\end{equation*}
We then introduce the symbol $\ell_0 \in S^0(\R^2,\mathscr{L}(B^{2,4}(\R),\C))\cap S^0(\R^2,\mathscr{L}(L^2(\R),\C))$ defined by
\begin{equation*}
    \forall (x,\xi)\in\R^2,\quad \ell_0^k(x,\xi) = \langle \mathring{u}_k(x,\xi),\cdot\rangle_{L^2(\R)}.
\end{equation*}
This allows us to apply \cite[Theorem 1.5]{BFKRV} to our setting.

\begin{proposition}
    \label{prop:elloperators}
    For each $k\in\N$, there exists an admissible symbol $\ell_\hbar^k \in S^0(\R^2,\mathscr{L}(B^{2,4}(\R),\C))\cap S^0(\R^2,\mathscr{L}(L^2(\R),\C))$ satisfying $\ell_\hbar^k = \ell_0^k + \O(\hbar)$ such that 
    \begin{equation*}
        (\ell_\hbar^k)^w \circ(\ell_\hbar^k)^{w,*} = {\rm Id}_{L^2(\R,\C)} + \O(\hbar^\infty)\quad\mbox{and}\quad (\ell_\hbar^k)^{w,*}\circ (\ell_\hbar^k)^w = (\Pi_\hbar^k)^w + \O(\hbar^\infty).
    \end{equation*}
\end{proposition}

\subsubsection{Dimensional reduction}

We now fix an energy level $E\in (-\infty,E_0)$ and, thanks to the introduction of the superadiabatic projectors $(\Pi_\hbar^k)_{k\in\N}$ and their corresponding partial isometries 
$(\ell_\hbar^k)_{k\in\N}$, we answer the question of finding one-dimensional scalar semiclassical pseudodifferential operators whose union of spectra coincide 
with the spectrum of $\mathcal{N}_\hbar$ in the interval $I = (-\infty,E)$. To this end, we introduce the integer $k_E \in\N$, defined as the largest integer $k$ for which 
the set $\{\mathring{\mu}_k \leq E\}$ is non-empty. We consider the symbol $\Pi_\hbar$ defined by 
\begin{equation*}
    \Pi_\hbar = \Pi_\hbar^1+\cdots+\Pi_\hbar^{k_E}.
\end{equation*}
We have 
\begin{equation}
    \label{eq:Pisuperadiabproj}
    \Pi_\hbar^w = \Pi_\hbar^w \circ \Pi_\hbar^w +\O(\hbar^\infty)\quad\mbox{and}\quad [\mathcal{N}_\hbar,\Pi_\hbar^w] =\O(\hbar^\infty).
\end{equation}
Moreover, Proposition \ref{prop:commutespectral} implies that there exists a family of real numbers $(\eps_\hbar)_{\hbar\in(0,\hbar_0)}$ satisfying $\eps_\hbar =\O(\hbar^\infty)$ such that, for any 
$\psi \in {\rm ran}(\mathbb{1}_{(-\infty,E)}(\mathcal{N}_\hbar))$, we have 
\begin{equation}
    \label{eq:Pihbarprojspectral}
    \Pi_\hbar^w \psi = \psi + \eps_\hbar \|\psi\|_{L^2}.
\end{equation}

We are now ready to state and prove the following theorem.

\begin{theorem}
    \label{thm:effectivespectrum}
    Let $E\in(-\infty,E_0)$. The spectrum of $\mathcal{N}_\hbar$ in $I = (-\infty,E)$ coincides modulo $\O(\hbar^\infty)$ with the one 
    of the bounded operator 
    \begin{equation*}
        \begin{bmatrix}
            \mathring{\mu}_{\hbar,1}^{w} & 0 & \cdots & 0 \\
            0 & \mathring{\mu}_{\hbar,2}^{w} &  & \vdots \\
            \vdots &  &  \ddots & 0 \\
            0 & \cdots & 0 & \mathring{\mu}_{\hbar,k_E}^{w}
         \end{bmatrix},
    \end{equation*}
    acting on $L^2(\R,\C^{k_E}$), where for $k\in\{1,\dots,k_E\}$, $\mathring{\mu}_{\hbar,k}^w$ is a bounded $\hbar$-semiclassical pseudodifferential operator 
    with symbol in $S^0(\R^2)$ and defined by 
    \begin{equation*}
        \mathring{\mu}_{\hbar,k}^w = (\ell_{\hbar}^k)^w \circ \mathcal{N}_\hbar \circ (\ell_{\hbar}^k)^{w,*}.
    \end{equation*}
\end{theorem}

As an immediate corollary, we get the following (rough) Weyl estimate on the number of discrete eigenvalues, used 
in the proof of Proposition \ref{prop:equivspectra}.

\begin{corollary}
    \label{cor:roughweyllaw2}
    For $E\in(-\infty,E_0)$, we denote by $N(\mathcal{N}_\hbar,E)$ the number of discrete eigenvalues lying in $(-\infty,E)$. Then, 
    for some constant $C>0$, it satisfies 
    \begin{equation*}
        N(\mathcal{N}_\hbar,E) \leq C \hbar^{-1}.
    \end{equation*}
\end{corollary}

The following proof is inspired by \cite[Corollary 1.6]{BFKRV}.

\begin{proof}[Proof of Theorem \ref{thm:effectivespectrum}]
    We introduce the symbol $\ell_\hbar$ defined by 
    \begin{equation*}
        \ell_\hbar = (\ell_\hbar^1,\dots,\ell_\hbar^{k_E}).
    \end{equation*}
    We easily check that this symbol is in the symbol class 
    \begin{equation*}
        S^0(\R^2,\mathscr{L}(B^{2,4}(\R),\C^{k_E}))\cap S^0(\R^2,\mathscr{L}(L^2(\R),\C^{k_E})).
    \end{equation*}
    By Proposition \ref{prop:elloperators} and by the property of two-by-two almost orthogonality of the superadiabatic projectors, we have 
    \begin{equation*}
        \ell_\hbar^w \circ\ell_\hbar^{w,*} = {\rm Id}_{L^2(\R,\C^{k_E})} + \O(\hbar^\infty)\quad\mbox{and}\quad \ell_\hbar^{w,*}\circ \ell_\hbar^w = \Pi_\hbar^w + \O(\hbar^\infty).
    \end{equation*}
    Also, we denote by $\mathcal{N}_{\hbar, \rm adiab}$ the operator defined by the matrix introduced in the statement of Theorem \ref{thm:effectivespectrum}. 
    Then, we observe that we can write
    \begin{equation*}
        \mathcal{N}_{\hbar, \rm adiab} = \ell_\hbar^w \circ \mathcal{N}_\hbar \circ \ell_\hbar^{w,*} + \O(\hbar^\infty),
    \end{equation*}
    where the remainder is understood as a bounded operator on $L^2(\R)^{k_E}$. Indeed, by the composition rule of operator-valued semiclassical pseudodifferential oeprators,
    the principal symbol of the right-hand side is given by the matrix 
    \begin{equation*}
        {\rm diag}(\mathring{\mu}_{1},\dots,\mathring{\mu}_{k_E}),
    \end{equation*}
    while the diagonal structure for the higher order terms follows from the two-by-two orthogonality (up to $\O(\hbar^\infty)$) of the operators $\Pi_\hbar^k$, $k\in\{1,\dots,k_E\}$.

    By simple considerations on the scalar semiclassical pseudodifferential 
    operators $(\mathring{\mu}_{\hbar,k})$, the spectrum of $\mathcal{N}_{\hbar, \rm adiab}$ is purely discrete in $(-\infty,E)$. Thus it makes sense to compare its spectrum 
    with the one of $\mathcal{N}_\hbar$.
    First observe that for $\lambda \in \sigma(\mathcal{N}_\hbar)\cap (-\infty,E)$, writing $\psi$ an associated eigenvector, by considering the vector 
    \begin{equation*}
        \phi = \ell_\hbar^w \psi \in L^2(\R,\C),
    \end{equation*}
    we have 
    \begin{equation*}
        \mathcal{N}_{\hbar,\rm adiab}\phi = \ell_\hbar^w \circ\mathcal{N}_\hbar \circ\Pi_\hbar^w \psi + \O(\hbar^\infty)\|\psi\|_{L^2(\R^2)},
    \end{equation*}
    since $\ell_\hbar^{w,*}\circ \ell_\hbar^w = \Pi_\hbar^w + \O(\hbar^\infty)$. By Equation \eqref{eq:Pihbarprojspectral}, since $\psi$ is in ${\rm ran}(\mathbb{1}_{(-\infty,E)}(\mathcal{N}_\hbar))$, we have $\Pi_\hbar^w\psi = \psi + \eps_\hbar\|\psi\|_{L^2(\R^2)}$. Thus, we find 
    \begin{equation}
        \label{eq:quasimodeadiab}
        \mathcal{N}_{\hbar,\rm adiab}\phi = \ell_\hbar^w (\lambda \psi) + (\eps_\hbar+\O(\hbar^\infty))\|\psi\|_{L^2(\R^2)} = \lambda \phi +\tilde{\eps}_\hbar\|\phi\|_{L^2(\R,\C^{k_E})},
    \end{equation}
    where the last equality follows from the relation $\ell_\hbar^w\circ \ell_\hbar^{w,*} = {\rm Id}_{L^2(\R,\C^{k_E})} + \O(\hbar^\infty)$ and where $\tilde{\eps}_\hbar = \O(\hbar^\infty)$ is an $\hbar$-scalar family independent of $\psi$.
    By the spectral theorem, we obtain that ${\rm dist}(\lambda,\sigma(\mathcal{N}_{\hbar,\rm adiab})) = \O(\hbar^\infty)$.
    
    Reciprocally, starting from an eigenpair $(\lambda,\phi)$ of $\mathcal{N}_{\hbar,\rm adiab}$ with $\lambda \in (-\infty,E)$, by considering this time the vector 
    \begin{equation*}
        \psi = \ell_\hbar^{w,*}\phi,
    \end{equation*}
    similar considerations lead to 
    \begin{equation*}
        {\rm dist}(\lambda,\mathcal{N}_\hbar) = \O(\hbar^\infty).
    \end{equation*}
    
    Thus, to show that the spectra coincide, we are left once again to address the problem of multiplicity.
    Following Definition \ref{def:spectracoincide}, we consider a $\hbar$-dependent interval $J_\hbar \subset (-\infty,E)$ and we write 
    $J_\hbar\cap\sigma(\mathcal{N}_\hbar) = \{\lambda_1, \dots ,\lambda_p\}$ (where the $\lambda_j$ are distinct). 
    We also consider the associated eigenspaces $(V_j)_{1\leq j\leq p}$.
    
    We introduce, for $j\in\{1,\dots,p\}$ the subspaces $\tilde{V}_j$ defined by
    \begin{equation*}
        \tilde{V}_j = \ell_\hbar^w(V_j).
    \end{equation*}
    By the previous considerations, the spaces $\tilde{V}_j$ are composed of quasimodes for the operator $\mathcal{N}_{\hbar,\rm adiab}$ for the corresponding eigenvalue $\lambda_j$.
    Moreover, we observe that since we have the relation $\ell_\hbar^{w,*} \circ \ell_\hbar^w = \Pi_\hbar^w + \O(\hbar^\infty)$, by Equation \eqref{eq:Pihbarprojspectral} we can write 
    \begin{equation*}
        \dim \tilde{V}_j = \dim V_j,
    \end{equation*}
    as soon as $\hbar$ is small enough, i.e. if $\hbar \leq \hbar_0$ for some $\hbar_0>0$. Now, observe that since the family $(\tilde{\eps}_\hbar)_{\hbar\in(0,\hbar_0)}$ only depends on the energy level $E$, Equation \eqref{eq:quasimodeadiab} readily implies
    \begin{equation}
        \label{eq:quasimodefull2}
        \|(\oplus_{j=1}^p \mathcal{N}_{\hbar,\rm adiab} - \lambda)\phi\| \leq \tilde{\eps}_\hbar \|\phi\|,
    \end{equation}
    for all $\phi = (\phi_1,\dots,\phi_p)\in \bigoplus_{j=1}^p \tilde{V}_j$, and where $\lambda = (\lambda_1,\dots,\lambda_p)$ and, we recall, 
    $\tilde{\eps}_\hbar = \O(\hbar^\infty)$.
    Setting $J_\hbar = [a_\hbar,b_\hbar]$, we consider the family of intervals $\tilde{J}_\hbar = [a_h-\tilde{\eps}_\hbar,b_h+\tilde{\eps}_\hbar]$. We wish to prove that we have 
    \begin{equation}
        \label{eq:spectruminclusion2}
        {\rm rank} \mathbb{1}_{J_\hbar}(\mathcal{N}_\hbar) \leq {\rm rank} \mathbb{1}_{\tilde{J}_\hbar}(\mathcal{N}_{\hbar,\rm adiab}).
    \end{equation}
    If we had ${\rm rank} \mathbb{1}_{\tilde{J}_\hbar}(\mathcal{N}_{\hbar,\rm adiab}) < {\rm rank} \mathbb{1}_{J_\hbar}(\mathcal{N}_\hbar)$, then the projection 
    $P : \bigoplus_{j=1}^p \tilde{V}_j \rightarrow {\rm ran} \mathbb{1}_{\tilde{J}_\hbar}(\mathcal{N}_{\hbar,\rm adiab})$ could not be injective: considering a non-zero $\psi$ in its kernel, the spectral theorem would give $\|(\oplus_{j=1}^p \mathcal{N}_{\hbar,\rm adiab} - \lambda)\psi\| > \tilde{\eps}_\hbar \|\psi\|$,
    which is in contradiction with Equation \eqref{eq:quasimodefull2}, as $\psi\neq 0$. Therefore, inequality \eqref{eq:spectruminclusion2} holds.

    The reciprocal inclusion follows from similar considerations, considering the spaces of quasimodes given by the (almost) partial isometry $\ell_\hbar^{w,*}$.
\end{proof}

\subsection{Study of the effective Hamiltonians}
\label{subsect:onedimensionpseudos}

We are left with the analysis of the effective Hamiltonians $\mathring{\mu}_{\hbar,k}$, $k\in\{1,\dots,k_E\}$, and with a study of their spectrum.

\begin{proposition}
    \label{prop:effectiveham}
    With the same notation as in Theorem \ref{thm:effectivespectrum}, for $k\in\{1,\dots,k_E\}$, the symbol $\mathring{\mu}_{\hbar,k}$ is admissible, i.e. it enjoys a semiclassical asymptotic expansion in $S^0(\R^2)$ as follows
    \begin{equation*}
        \mathring{\mu}_{\hbar,k} \sim \mathring{\mu}_{k} + \sum_{j\geq 1}\hbar^j \mathring{\mu}_{j,k},\quad \mbox{with}\quad \forall j\geq 1,\ \mathring{\mu}_{j,k}\in S^0(\R^2).
    \end{equation*}
    Moreover, the subprincipal symbol $\mathring{\mu}_{1,k}$ is given in the compact region $K_E$ by 
    \begin{equation}
        \label{eq:expressionsubprinc}
        \mathring{\mu}_{1,k} = 2k\langle \mathring{u}_k, p_0^2 \ct \,\mathring{u}_k\rangle- 2 \mathring{\kappa}\langle \mathring{u}_k, p_0 \ct^3 \mathring{u}_k\rangle +  {\rm Im} \langle \mathring{u}_k, \{n_0,\mathring{u}_k\}\rangle + \mu_k {\rm Im} \langle \partial_x \mathring{u}_k,\partial_\xi \mathring{u}_k\rangle.
    \end{equation}
\end{proposition}

\begin{proof}
    The fact that the symbol $\mathring{\mu}_{\hbar,k}$ is admissible, that is, that the symbols $\mathring{\mu}_{j,k}$ are in the symbol class $S^0(\R^2)$ rather than $S^{-j\eta}(\R^2)$, $j\geq 1$, as one would expect from the asymptotic expansion of the symbols $n_\hbar$, $\Pi_\hbar^k$ and $\ell_\hbar^k$, comes from the exponential decay of the eigenvalues of the Montgomery operators. Indeed, based on \cite{BFKRV}[Remark 3.1], it can be shown that the symbol $\ell_\hbar^k$ is, at each point in the phase space $\R^2$, a linear form on $L^2(\R_{\ct})$ associated with an exponentially decaying function in $L^2(\R_{\ct})$. In this way, the formula for the composition of pseudodifferential operators (see Proposition \ref{prop:symb_cal}) shows that the presence of the cutoffs $\theta(\hbar^\eta\cdot)$ in the definition of the symbol $n_\hbar$ becomes redundant. Since the truncated functions appearing in the definition of $\mathcal{N}_\hbar$ are bounded (as well as all their derivatives) by design, the asymptotic expansion of $\mathring{\mu}_{\hbar,k}$ is in the desired class of symbols.
    
    We are left to compute the subprincipal symbol $\mathring{\mu}_{1,k}$. By Proposition \ref{prop:symb_cal}, it is given by
    \begin{equation*}
        \mathring{\mu}_{1,k} = \ell_0^k n_1 \ell_0^{k,*} + \frac{1}{2i}\{\ell_0^k, n_0\} \ell_0^{k,*} + \frac{1}{2i} \ell_0^k \{n_0,\ell_0^{k,*}\}+ \ell_1^k n_0 \ell_0^{k,*} + \ell_0^k n_0 \ell_1^{k,*}.
    \end{equation*}
    By the definition of $\ell_0^k$, the equation above can be rewritten as follows
    \begin{equation*}
        \mathring{\mu}_{1,k} = \langle \mathring{u}_k, n_1 \mathring{u}_k \rangle +  {\rm Im} \langle \mathring{u}_k, \{n_0,\mathring{u}_k\}\rangle+ (\ell_1^k n_0 \ell_0^{k,*} + \ell_0^k n_0 \ell_1^{k,*}),
    \end{equation*}
     where $\{n_0,\mathring{u}_k\}(x,\xi) = \partial_\xi n_0\partial_x \mathring{u}_k-\partial_x n_0\partial_\xi \mathring{u}_k$. 
    
    The first term of this sum can be easily computed thanks to Lemma \ref{lem:Nhbarpseudo}. 
    To identify the third term, we introduce the notation $\ell_k^1 = \langle \mathring{u}_k^1,\cdot\rangle$. This allows us to rewrite the term under consideration as
    \begin{equation*}
        \ell_1^k n_0 \ell_0^{k,*} + \ell_0^k n_0 \ell_1^{k,*} = 2 \mu_k \Re \langle \mathring{u}_k, \mathring{u}_k^1\rangle.
    \end{equation*}
    To determine this scalar product, we make use of \cite{BFKRV}[Lemma 2.4] and the compatibility relation therein satisfied by $\ell_1^k$:
    \begin{equation*}
        \ell_0^k \ell_1^{k,*} + \ell_1^k\ell_0^{l,*} +\frac{1}{2i}\{\ell_0^k,\ell_0^{k,*}\} = 0,
    \end{equation*}
    which can be rewritten as
    \begin{equation*}
        \Re \langle \mathring{u}_k, \mathring{u}_k^1\rangle = \frac{1}{2}{\rm Im} \langle \partial_x \mathring{u}_k,\partial_\xi \mathring{u}_k\rangle,
    \end{equation*}
    and this concludes the proof.
\end{proof}

Combining the results of Theorem \ref{thm:effectivespectrum} and Proposition \ref{prop:effectiveham}, one obtains Theorem \ref{thm:main}.
We now give the proof of Corollary \ref{cor:bottomsemiexcited}.
 
\begin{proof}[Proof of Corollary \ref{cor:bottomsemiexcited}]
    This result is a direct application of the Birkhoff normal form procedure implemented in \cite{Sjo1}, or equivalently of the statement \cite{DimSj}[Theorem 14.9]. Indeed, under Assumption \ref{assum:uniquewell}, the principal symbol $\mathring{\mu}_1$ has a unique minimum attained at $(0,\nu_c)$. Moreover, this minimum is nondegenerate as a straightforward computation using Equation \eqref{eq:smootheigenvalues} shows that the Hessian of $\mathring{\mu}_1$ at this critical point is diagonal with eigenvalues given by 
    \begin{equation*}
        \frac{2}{3}\frac{\delta''(0)}{\delta_c^{1/3}}\tilde{\mu}_{1,c}\quad\mbox{and}\quad \partial_\nu^2\tilde{\mu}_1(\nu_{1,c}).
    \end{equation*}

    The structure of the spectrum described by Equation \eqref{eq:asympgenstuct1} is then given by \cite{DimSj}[Theorem 14.9]. However, to obtain the more precise asymptotics of Equation \eqref{eq:asymptprecise}, we need to give some details about the functions $f_0$ and $f_1$.

    The first step in the Birkhoff normal form procedure is to build the function $f_0$ such that it satisfies
    \begin{equation*}
        f_0(\sigma) = \delta_c^{2/3}\tilde{\mu}_{1,c} + \frac{C_1}{2}\sigma + \O(\sigma^2),
    \end{equation*}
    with $C_1$ being the geometric mean of the eigenvalues of the Hessian of $\mathring{\mu}_1$. We then readily observe that
    \begin{equation*}
        C_1 = 2 \delta_c^{2/3} c_1.
    \end{equation*}

    Thus, to complete the proof of Corolloray \ref{cor:bottomsemiexcited}, we are left with computing $f_1(0)$: in the Birkhoff normal form, this value coincides with the evaluation of the subprincipal symbol $\mathring{\mu}_{1,1}$ at the bottom of the well, that is, at $(0,\nu_c)$.

    First, observe that by a simple use of Equations \eqref{eq:n0rescaling} and \eqref{eq:smootheigenfunction}, we have
    \begin{equation*}
        2k(0)\langle \mathring{u}_1(0,\nu_c), p_0^2 \ct \,\mathring{u}_1(0,\nu_c)\rangle- 2 \mathring{\kappa}(0)\langle \mathring{u}_1(0,\nu_c), p_0 \ct^3 \mathring{u}_1(0,\nu_c)\rangle = \delta_c^{2/3}\langle L \tilde{u}_{1,c},\tilde{u}_{1,c}\rangle,
    \end{equation*}
    where the operator $L$ is defined in Corollary \ref{cor:bottomsemiexcited}.
     
    We are left to prove that the other two terms of $\mathring{\mu}_{1,1}$ in the expression \eqref{eq:expressionsubprinc} vanish at the bottom of the well. To do so, we derive a more precise expression of these two terms. 
    
    To treat the term $\langle \mathring{u}_k, \{n_0,\mathring{u}_k\}\rangle$, we again use Equation \eqref{eq:n0rescaling}, which we recall here:
    \begin{equation*}
        n_0(x,\xi) = \mathring{\delta}(x)^{2/3} \mathcal{T}^*(x) \mathfrak{M}(\Theta(x,\xi)) \mathcal{T}(x),
    \end{equation*}
    where $\Theta(x,\xi) = \xi\mathring{\delta}(x)^{-1/3}$.
    This formula allows to write
    \begin{equation*}
        \partial_\xi n_0 = \mathring{\delta}^{1/3} \mathcal{T}^* \left(\partial_\nu \mathfrak{M}\circ \Theta\right) \mathcal{T},
    \end{equation*}
    and 
    \begin{equation*}
        \partial_x n_0 = -\frac{\mathring{\delta}'}{3\mathring{\delta}} \Theta\,\mathcal{T}^* \left(\partial_\nu \mathfrak{M}\circ \Theta\right) \mathcal{T}+ \partial_x \mathcal{T}^* \mathcal{T} n_0 + n_0 \mathcal{T}^*\partial_x \mathcal{T},
    \end{equation*}
    with similar formulae for the derivatives of $u_k$.
    Combined together and taking the scalar product with $u_k$, one is left with
    \begin{equation*}
         -\frac{\mathring{\delta}'(x)}{3 \mathring{\delta}(x)^{2/3}}\left\langle\tilde{u}_k(\nu),\partial_\nu\mathfrak{M}(\nu)\mathcal{I}\tilde{u}_k(\nu) - \left(\mathcal{I}\mathfrak{M}(\nu)+ \mathfrak{M}(\nu)\mathcal{I}^*\right)\partial_\nu \tilde{u}_k(\nu)  \right\rangle,
    \end{equation*}
    evaluated at $\nu = \Theta(x,\xi)$ and where $\mathcal{I}$ is given by
    \begin{equation*}
        \mathcal{I} = \frac{1}{2}+ \ft\partial_{\ft}.
    \end{equation*}

    \vspace{0.25cm}

    Finally, to treat the term ${\rm Im} \langle \partial_x \mathring{u}_k,\partial_\xi \mathring{u}_k\rangle$, we can use the same formulae as before to obtain
    \begin{equation*}
        \begin{aligned}
        {\rm Im} \langle \partial_x \mathring{u}_k,\partial_\xi \mathring{u}_k\rangle 
        &= {\rm Im} \langle \mathcal{T}\partial_x \mathcal{T}^* \tilde{u}_k(\Theta), \mathring{\delta}^{-1/3} \partial_\nu \tilde{u}_k(\Theta)\rangle\\
        &= -\frac{\mathring{\delta}'}{3 \mathring{\delta}^{4/3}}{\rm Im}\langle\mathcal{I}\tilde{u}_k(\Theta),\partial_\nu \tilde{u}_k(\Theta)\rangle.
        \end{aligned}
    \end{equation*}
    Thus, evaluating at the bottom of the well, we find the desired quantity, and this concludes the proof.
    % Compiling every expression, one is left with
    % \begin{multline*}
    %     \mathring{\mu}_{1,k} = 2k\langle u_k, p_0^2 \ct \,u_k\rangle- 2 \mathring{\kappa}\langle u_k, p_0 \ct^3 u_k\rangle \\
    %     -\frac{\delta'}{3 \delta^{2/3}} \left(\tilde{\mu}(\xi \delta^{-1/3}){\rm Im}\langle\tilde{u}_k,\partial_\nu\tilde{u}_k\rangle- 2{\rm Im} \langle \tilde{u}_k(\xi \delta^{-1/3}),((\xi \delta^{-1/3})^2-4\xi \delta^{-1/3}\ft^2+\frac{5}{4}\ft^4)\partial_\nu \tilde{u}_k(\xi \delta^{-1/3})\rangle\right)\\
    %     -\frac{\delta'}{3 \delta^{4/3}}\mu_k {\rm Im}\langle\mathcal{I}\tilde{u}_k(\xi \delta^{-1/3}),\partial_\nu \tilde{u}_k(\xi \delta^{-1/3})\rangle.
    % \end{multline*}
\end{proof}

%%%%%%%%%%%%%%%%%%%%%%%%%%%%%%%%%%%%%%%%%%%%%%%%%%%%%%%%%%%%%%%%%%%%%%%%%%%%%%%%%%%%%%%%%%%%%%%%%%%%%%%%%%%%%%%%%%%%%%%%%%%%%%%%%%%%%%%%%%%%%%%%%%%%%%%%%%%%%%%%%%%%%%%%%%%%%%

\appendix
\section{Pseudodifferential calculus with operator-valued symbols}
\label{sect:pseudoropvalued}

We review here elements of pseudodifferential calculus with operator-valued symbols that we use along this article. The reader can refer to~\cite{Zwobook} for proofs in the scalar case that can be adapted to our setting, as in~\cite{Ker}.  
\smallskip 

In this section, we denote a point $(x,\xi)\in\R^{2d}$ by $X$. Consider a family of Hilbert spaces $\mathscr{A}_{X}$ indexed by $X\in\R^{2d}$, and satisfying the following properties:
\begin{itemize}
    \item[(i)] $\mathscr{A}_{X}$ = $\mathscr{A}_{Y}$ as vector spaces for all $X,Y\in\R^{2d}$, denoted simply by $\mathscr{A}$,
    \item[(ii)] There exist $N_0 \geq 0$ and $C >0$ such that
    $\Vert u \Vert_{\mathscr{A}_X} \leq C\langle X-Y\rangle^{N_0}\Vert u \Vert_{\mathscr{A}_Y}$ for all $u\in\mathscr{A},\,X,Y\in\R^{2d}$.
\end{itemize}

Let $\mathscr{B}_{X},\,X\in\R^{2d}$ be another family of such Hilbert spaces. The class of symbols we are interested in is defined as follows: for $\eta\in\R$,
we say that $p \in C^{\infty}(\R^{2d},\mathscr{L}(\mathscr{A},\mathscr{B}))$ belongs to $S^\eta(\R^{2d},\mathscr{L}(\mathscr{A}_X,\mathscr{B}_X))$ if for every 
$\alpha \in \N^{2n}$, we can find a constant $C_{\alpha}$ such that
\begin{equation}
    \Vert \partial_{X}^{\alpha}p\Vert_{\mathscr{L}(\mathscr{A}_X,\mathscr{B}_X)} \leq h^\eta C_{\alpha},\quad X \in \R^{2d}.
\end{equation}
The symbol $p$ can depend on the semiclassical parameter $h \in (0,h_0]$, as long as the precedent inequalities are uniform for $h$ small. In particular, we have the following characterization. 

\begin{proposition}
    Let $p: \R^{2d} \longrightarrow \mathscr{L}(\mathscr{A},\mathscr{B})$, then $p$ is in the symbol class $S^0(\R^{2d},\mathscr{L}(\mathscr{A}_X,\mathscr{B}_X))$ if and only if,
    \begin{enumerate}
        \item For all $(a,b) \in \mathscr{A}_0 \times \mathscr{B}_0$, the application defined for $X \in\R^{2d}$ by
        \begin{equation*}
            p^{(a,b)}(X) := \langle p(X)a,b\rangle
        \end{equation*}
        is smooth.
        \item For all $\alpha\in\N^{2n}$, there exists $C_{\alpha}$ such that, for all $(a,b) \in \mathscr{A}_0 \times \mathscr{B}_0$,
        \begin{equation*}
            \lvert \partial_{X}^{\alpha}p^{(a,b)} \rvert \leq C_{\alpha}\Vert a\Vert_{\mathscr{A}_X}\Vert b\Vert_{\mathscr{B}_X}.
        \end{equation*}
    \end{enumerate}
\end{proposition}

When a symbol $p$ depends on the semiclassical parameter, we consider asymptotic expansions in this parameter according to the following definition.

\begin{definition}[Admissible symbols] \label{def:classicalsymbols} Let $\left( p_j \right)_{j \in \N}$ a sequence of symbols in $S^0(\R^{2d},\mathscr{L}(\mathscr{A}_X,\mathscr{B}_X))$. We say that $p \in S^0(\R^{2d},\mathscr{L}(\mathscr{A}_X,\mathscr{B}_X))$ has as \textit{semiclassical asymptotic expansion}  $\sum_{j\geq 0} h^jp_j$ if for all $N \in \N$, there exists $r_N \in S^0(\R^{2d},\mathscr{L}(\mathscr{A}_X,\mathscr{B}_X))$ such that
\[
p - \sum_{j=0}^N h^j p_j = h^{N+1} r_N.
\]
We then simply write
\[
p \sim \sum_{j = 0}^{+\infty} h^j p_j \quad \mbox{in} \quad S^0(\R^{2d},\mathscr{L}(\mathscr{A}_X,\mathscr{B}_X)).
\]
\end{definition}

Reciprocally, asymptotic expansions define symbols as stated in the following proposition. 

\begin{proposition}[Borel Summation] \label{prop:Borelsum} Let $\left( p_j \right)_{j \in \N}$ a sequence of symbols in $S(\R^{2d},\mathscr{L}(\mathscr{A}_X,\mathscr{B}_X))$, then there exists a symbol $p \in S^0(\R^{2d},\mathscr{L}(\mathscr{A}_X,\mathscr{B}_X))$ unique up to $\O(h^\infty)$ such that
\[
p \sim \sum_{j = 0}^{+\infty} h^j p_j \quad \mbox{in} \quad S^0(\R^{2d},\mathscr{L}(\mathscr{A}_X,\mathscr{B}_X)).
\]
\end{proposition}

We will use the semiclassical Weyl quantization that we recall in the next definition.  

\begin{definition}
    Let $p \in S^0(\R^{2d},\mathscr{L}(\mathscr{A}_X,\mathscr{B}_X))$ a symbol, we will denote by $p^w$ or ${\rm op}_h(p)$ its Weyl quantization, that is, the operator defined for $u\in \mathscr{S}(\R^{d},\mathscr{A})$ by
    \begin{equation}\label{def:pseudo}
        p^w u(x) = \frac{1}{(2\pi h)^n}\iint_{\R^{2d}} e^{i\frac{(x-y)\cdot \xi}{h}}p\left(\frac{x+y}{2},\xi\right)u(y)\,dyd\xi.
    \end{equation}
\end{definition}

We have the following continuity result.

\begin{proposition}
    Let $p \in S^0(\R^{2d},\mathscr{L}(\mathscr{A}_X,\mathscr{B}_X))$. Then $p^w$ is uniformly continuous from $\mathscr{S}(\R^{d},\mathscr{A})$ to $\mathscr{S}(\R^{d},\mathscr{B})$, and extends as a continuous operator from $\mathscr{S}'(\R^{d},\mathscr{A})$ to $\mathscr{S}'(\R^{d},\mathscr{B})$.
\end{proposition}

The Calderòn-Vaillancourt theorem extends to semiclassical pseudodifferential operators with operator-valued symbols.

\begin{proposition}
    Assume $\mathscr{A}_X = \mathscr{A}$, $\mathscr{B}_X = \mathscr{B}$ for all $X\in\R^{2d}$. If $p \in S^0(\R^{2d},\mathscr{L}(\mathscr{A},\mathscr{B}))$, then $p^w$ extends to a bounded operator
    \begin{equation*}
        L^{2}(\R^{d},\mathscr{A}) \longrightarrow L^{2}(\R^{d},\mathscr{B}).
    \end{equation*}
\end{proposition}

We also have the usual composition stability and formulae of pseudodifferential calculus. Let $\mathscr{C}_X,\,X\in\R^{2d}$ be a third family of Hilbert spaces.

\begin{proposition}[Symbolic calculus]\label{prop:symb_cal}
    Consider 
    \[
    p \in S¨0(\R^{2d},\mathscr{L}(\mathscr{B}_X,\mathscr{C}_X)\;\;\mbox{and}\;\;q \in S^0(\R^{2d},\mathscr{L}(\mathscr{A}_X,\mathscr{B}_X)).
    \]
    Then $p^w\circ q^w= r^w_h$, where $r \in S^{0}(\R^{2d},\mathscr{L}(\mathscr{A}_X,\mathscr{C}_X))$ is given by 
    \begin{equation}
        r = \left.{\rm exp}\left(\frac{ih}{2}\sigma(D_x,D_{\xi};D_y,D_{\eta})\right)(p(x,\xi)q(y,\eta))\right|_{x=y,\xi = \eta},
    \end{equation}
    where $\sigma$ is the usual symplectic form.
\end{proposition}

We recall the notion of ellipticity for operator-valued symbols. 

\begin{definition}
    \label{def:symelliptic}
	Let $p\in S^0(\R^{2d},\mathscr{L}(\mathscr{A},\mathscr{B}))$. The symbol $p$ is said to be \emph{elliptic} in $S^0(\R^{2d},\mathscr{L}(\mathscr{A},\mathscr{B}))$ if there exists $C>0$ such that, for all $X\in\R^{2d}$, $p(X)$ is invertible and 
	\begin{equation*}
		\Vert p(X)^{-1} \Vert_{\mathscr{L}(\mathscr{A},\mathscr{B})} \leq C.
	\end{equation*}
\end{definition}

The ellipticity condition on a symbol allows for the construction of a parametrix to the associated quantized operator.

\begin{theorem}
	\label{thm:appelliptic}
	Let $p$ be an elliptic symbol in $S^0(\R^{2d},\mathscr{L}(\mathscr{A},\mathscr{B}))$, then there exists $h_0>0$ such that, for all $h\in(0,h_0)$, the operator $P = p^w$ is invertible in $\mathscr{L}(L^2(\R^d,\mathscr{A}),L^2(\R^d,\mathscr{B}))$ and there exists $q\in S^0(\R^{2d},\mathscr{L}(\mathscr{B},\mathscr{A}))$
	such that 
	$	P^{-1} = q^w$,
	and $q = p^{-1} + h r$ with $r\in S^0(\R^{2d},\mathscr{L}(\mathscr{B},\mathscr{A}))$.
\end{theorem}

To conclude this section, we recall the existence of a Gärding inequality for semiclassical pseudodifferential operators with operator-valued symbols.

\begin{theorem}\label{thm:garding}
    Suppose that the Hilbert space $\mathscr{A}$ is continuously embedded in $\mathscr{B}$. For a symbol $p\in S^0(\R^{2d},\mathscr{L}(\mathscr{A},\mathscr{B}))$ whose values are positive self-adjoint operators on $\mathscr{B}$ with domain $\mathscr{A}$, i.e. satisfying
    \begin{equation*}
        \forall X\in\R^{2d}, \forall a \in \mathscr{A},\quad \langle p(X)a,a\rangle_{\mathscr{B}} \geq 0, 
    \end{equation*}
    Then, there exists a constant $C>0$ and $h_0>0$ such that for $h\in (0,h_0]$, we have 
    \begin{equation*}
        \langle p^w u, u\rangle_{L^2(\R^d,\mathscr{B})} \geq - C h \|u\|_{L^2(\R^d,\mathscr{A})}^2.
    \end{equation*}
\end{theorem}

\end{document}